\documentclass[leqno,twoside]{article}
\usepackage{amsmath,amssymb,amsthm}
\usepackage[margin=2.7cm]{geometry} 
\usepackage{titlesec,hyperref}
\usepackage{fancyhdr}

\usepackage{color}

\fancypagestyle{plain}
{
\fancyhf{}

}

\titleformat{\subsection}{\it}{\thesubsection.\enspace}{1pt}{}

\newtheorem{theo}{Theorem}[section]
\newtheorem{lemm}[theo]{Lemma}
\newtheorem{defi}[theo]{Definition}

\newtheorem{prop}[theo]{Proposition}
\newtheorem{rema}[theo]{Remark}
\numberwithin{equation}{section}

\allowdisplaybreaks

\def\la{\langle}
\def\ra{\rangle}
\def\ep{\varepsilon}
\def\G{G_\delta}
\def\Gd{G_\delta(t,\xi)}
\def\M{M_\delta}
\def\Md{M_\delta(t,\xi)}
\def\lm{\lesssim}
\def\p{\partial}
\def\a{\alpha}
\def\hf{\hat{f}}
\def\no{\nonumber}
\def\th2{\frac{\theta}{2}}

\begin{document}
\title{Gevrey Smoothing Effect for Solutions of the Non-Cutoff Boltzmann Equation in Maxwellian Molecules Case}

\author{Teng-Fei Zhang$^1$
\quad Zhaoyang Yin$^2$ \\[10pt]
Department of Mathematics, Sun Yat-sen University,\\
510275, Guangzhou, P. R. China.
}
\footnotetext[1]{Corresponding author. Email: \it fgeyirui@163.com}
\footnotetext[2]{Email: \it mcsyzy@mail.sysu.com.cn}

\maketitle
\hrule

\begin{abstract}
In this paper we study the Gevrey regularity for the weak solutions to the Cauchy problem of the non-cutoff spatially homogeneous Botlzmann equation for the Maxwellian molecules model with the singularity exponent $s\in (0,1)$. We establish that any weak solution belongs to the Gevrey spaces for any positive time.

\vspace*{5pt}

\vspace*{5pt}
\end{abstract}

\vspace*{10pt}


\tableofcontents
\section{Introduction}

\subsection{The Boltzmann equation}

In this paper we are concerned with the Cauchy problem of the non-cutoff Boltzmann equation. First we introduce the Cauchy problem of the full (or, spatially inhomogeneous) Boltzmann equation without angular cutoff, with a $T >0$,
\begin{align}\label{BE-full}
  \left\{
    \begin{array}{l}\displaystyle
     f_t(t,x,v)+v\cdot \nabla_xf(t,x,v)=Q(f,f)(v),\quad t\in (0,T], ~x\in \mathbb{T}^3,~v \in \mathbb{R}^3,\\
     f(0,x,v)=f_0(x,v).
    \end{array}
  \right.
\end{align}

Above, $f=f(t,x,v)$ describes the density distribution function of particles located around position $x\in \mathbb{T}^3$ with velocity $v\in \mathbb{R}^3$ at time $t \geq 0$. The right-hand side of the first equation is the so-called Boltzmann bilinear collision operator acting only on the velocity variable $v$:
\[
Q(g, f)=\int_{\mathbb{R}^3}\int_{\mathbb S^{2}}B\left({v-v_*},\sigma
\right)
 \left\{g'_* f'-g_*f\right\}d\sigma dv_*.
\]
Note that we use the well-known shorthands $f=f(t,x,v)$, $f_*=f(t,x,v_*) $, $f'=f(t,x,v') $, $f'_*=f(t,x,v'_*) $ throughout this paper.

Then, we consider the Cauchy problem of the Boltzmann equation in the spatially homogeneous case, that is, for a $T >0$,
\begin{align}\label{BE}
  \left\{
    \begin{array}{l}
      f_t(t,v)=Q(f,f)(v),\quad t\in (0,T], ~  v \in \mathbb{R}^3,\\
      f(0,v)=f_0(v),
    \end{array}
  \right.
\end{align}
where ``spatially homogeneous" means that $f$ depends only on $t$ and $v$.

By using the $\sigma$-representation, we can describe the
relations between the post- and pre-collisional velocities as
follows, for $\sigma \in \mathbb S^2$,
$$
v'=\frac{v+v_*}{2}+\frac{|v-v_*|}{2}\sigma,~ ~ v'_*
=\frac{v+v_*}{2}-\frac{|v-v_*|}{2}\sigma.
$$
We point out that the collision process satisfies the conservation of momentum and kinetic energy, i.e.
$$
v+v_*=v'+v'_*,\qquad  |v|^2+|v_*|^2=|v'|^2+|v'_*|^2.
$$

The collision cross section $B(z, \sigma)$ is a given non-negative function depending only on the interaction law between particles. From a mathematical viewpoint, that means $B(z, \sigma)$ depends only on the relative velocity $|z|=|v-v_*|$ and the deviation angle $\theta$ through the scalar product
$\cos \theta=\frac{z}{|z|} \cdot \sigma$.

The cross section $B$  is assumed here to be of the type:
$$
B(v-v_*, \cos \theta)=\Phi (|v-v_*|) b(\cos \theta),~~
\cos \theta=\frac{v-v_*}{|v-v_*|} \cdot \sigma,~~
0\leq\theta\leq\frac{\pi}{2},
$$
where, $\Phi$ stands for the kinetic factor which is of the form:
$$
\Phi (|v-v_*|) = |v-v_*|^{\gamma},
$$
and $b$ denotes the angular part with singularity such that,
$$
\sin \theta b(\cos \theta) \sim K\theta^{-1-2s}, \ \ \mbox{as} \ \ \theta\rightarrow 0+,
$$
for some positive constant $K$ and $0< s <1$.

We remark that if the inter-molecule potential satisfies specifically the inverse-power law $U(\rho) = \rho ^{-(p-1)} ~(\textrm{where } p>2) $, it holds $\gamma = \frac{p-5}{p-1} $, $ s=\frac{1}{p-1} $. Generally, the cases $\gamma >0$, $\gamma =0$, and $\gamma <0$ correspond to so-called hard, Maxwellian, and soft potential respectively. And the cases $0<s<1/2$, $1/2 \leq s<1$ correspond to so-called mild singularity and strong singularity respectively.

\subsection{Review of related references}

Now we give a brief review about some related researches. Firstly we refer the reader to Villani's review book \cite{Villani} for the physical background and the mathematical theories of the Boltzmann equation. And for more information about the non-cutoff theories, one can consult Alexandre's review paper \cite{Alex-review}.

Before continuing the statement, we provide the definition of Gevrey spaces $G^s(\Omega)$ where $\Omega$ is an open subset of $\mathbb{R}^3$. (It could be found in many references, e.g. \cite{Mori-Xu-09,Zhang-Yin}.)
\begin{defi}\label{Def_2}
For $0<s<+\infty$, we say that $ f \in G^s(\Omega) $, if $f \in C^\infty(\Omega)$, and there exist $C>0,~ N_0>0$ such that
$$ \|\partial^\alpha f\|_{L^2(\Omega)} \leq C^{|\alpha|+1} {\{\alpha! \}^s},\quad
\forall \alpha \in \mathbb{N}^3,~ |\alpha| \geq N_0.
$$

If the boundary of $\Omega$ is smooth, by using the Sobolev embedding theorem, we have the same type estimate with $L^2$ norm replaced by any $L^p$ norm for $2 < p \leq +\infty$. More specifically, on the whole space $\Omega = \mathbb{R}^3$, it is also equivalent to
$$
e^{c_0 (-\Delta)^{1/(2s)}} (\partial ^{\beta_0} f) \in L^2(\mathbb{R}^3),
$$
for some $c_0>0$ and  $\beta_0 \in \mathbb{N}^3$, where $e^{c_0 (- \Delta)^{1/(2s)}} $ is the Fourier multiplier defined by
$$
e^{c_0 (- \Delta)^{1/(2s)}} u(x)= \mathcal{F}^{-1}\left( e^{c_0 |\xi|^{1/s}}\textrm{\^{u}}(\xi) \right).
$$

When $s = 1$, it is usual analytic function. If $s > 1$, it is Gevrey class function. And for $0 < s < 1$, it is called ultra-analytic function.
\end{defi}

In 1984 Ukai showed in \cite{Ukai-84} that there exists a unique local solution to the Cauchy problem for the Boltzmann equation in Gevrey classes for both spatially homogeneous and inhomogeneous cases, under the assumption on the cross section:
\begin{align*}
&\big| B(|z|,\cos \theta) \big| \leq  K(1+|z|^{-\gamma'}+|z|^\gamma) \theta^{-n+1-2s},
      \quad n  \textrm{ is dimensionality},\\
&(0\leq \gamma' < n,~ 0\leq \gamma <2,~ 0\leq s<1/2, ~\gamma +6s<2 ).
\end{align*}

By introducing the norm of Gevrey space
$$
\|f\|^{U}_{\delta,\rho,\nu} = \sum_{\alpha} \frac{\rho^{|\alpha|}}{\{\alpha!\}^\nu}
\|e^{\delta \langle v \rangle^2} \partial_v^{\alpha} f \|_{L^\infty(\mathbb{R}^n_v)},
$$
Ukai proved that in the spatially homogeneous case, for instance, under some assumptions for $\nu$ and the initial datum $f_0(v)$, the Cauchy problem (\ref{BE}) has a unique solution $f(t,v)$ for $t\in (0,T]$.

In \cite{Desvillettes2} Desvillettes and Wennberg studied firstly the $C^\infty$ smoothing effect for solutions of Cauchy problem in spatially homogeneous non-cutoff case, and conjectured Gevrey smoothing effect. And later, Desvillettes et al. proved in \cite{Desvillettes} the propagation of Gevrey regularity for solutions for Maxwellian molecules case.

In 2009 Morimoto et al. considered in \cite{Mori-Ukai-Xu-Yang} the Gevrey regularity for the solutions to the Cauchy problem of the linearized Boltzmann equation, for the Maxwellian molecules model and around the absolute Maxwellian distribution, by virtue of the following mollifier:
$$
G_\delta (t,D_v)=\frac{e^{t \langle D_v \rangle^{1/\a}}}
                      {1+\delta e^{t \langle D_v \rangle^{1/\a}}},\quad 0< \delta <1.
$$
Therein the authors proved that the solutions belong to the Gevrey spaces $G^{1/\a}$ for any $0<\a<1$, when the singularity exponent $s \in (0,1)$.

In \cite{Lekrine-Xu-09} Lekrine and Xu proved that, using the same method, the Gevrey regularity for solutions to the Kac's equation (a simplification of the Boltzmann equation to one dimensional case), and Gevrey regularity for the radially symmetric weak solutions to the Boltzmann equation. Under the mild singularity assumption $s\in (0,1/2)$, they proved the radially symmetric weak solutions are in the Gevrey spaces $G^{1/(2s')}$ for any $s'\in (0,s)$ and any time $t>0$. Recently, Glangetas and Najeme complemented their results for the strong singularity case $s\in [1/2,1)$, and established the analytic smoothing effect in \cite{Glangetas}. The similar mollifier was used by Morimoto and Xu, to prove the ultra-analytic smoothing effect for spatially homogeneous nonlinear Landau equation and the linear and non-linear Fokker-Planck equations (see \cite{Mori-Xu-09}). And Lerner et al. proved in \cite{Morimoto-Starov-GS} that the Cauchy problem of the radially symmetric spatially homogeneous non-cutoff Boltzmann equation with Maxwellian molecules enjoys the same Gelfand-Shilov regularizing effect as the Cauchy problem of some kind of evolution equation associated to a fractional harmonic oscillator.

In the mild singularity case of $0<s<1/2$, Huo et al. proved in \cite{Huo} that any weak solution $f(t,v)$ to the Cauchy problem (\ref{BE}) satisfying the natural boundedness on mass, energy and entropy,
belongs to $H^{+\infty}(\mathbb{R}^n)$ for any $0<t \leq T$.

In 2010 Morimoto and Ukai considered in \cite{Mori-Ukai} the Gevrey regularity (precisely, $G^{1/(2s)}$), of $C^\infty$ solutions with the Maxwellian decay to the Cauchy problem of spatially homogeneous Boltzmann equation, with a modified kinetic factor $\Phi(v)=\la v \ra^\gamma$. Recently, Zhang and Yin extended this result for the general kinetic factor $\Phi(v)=|v|^\gamma$ (see \cite{Zhang-Yin}), and as a continuation of that, they studied the problem for the spatially inhomogeneous case, for extending to a larger range of the exponent $\gamma$, and for a special critical singularity case, respectively. Through these works we attempt to give an almost whole description of the Gevrey regularity for the so-called smooth Maxwellian decay solution. For more details, one can consult \cite{Zhang-Yin-2, Zhang-Yin-3}.

In this present work, we consider the Gevrey smoothing effect for the weak solutions to the Cauchy problem of spatially homogeneous Botlzmann equation without cut-off. Because of the difficulty coming from the interaction between the generally kinetic factor $\Phi(v)=|v|^\gamma$ and the mollifier operator defined below, we will restrict our attention to the Maxwellian molecules model $\Phi \equiv 1$. Further, we consider not only the mild singularity case $0<s<1/2$ but also strong singularity case $1/2 \le s <1$, the latter case of which, as is known to all, is difficulty to deal with and thus there are few of research about that. Additionally, we point out that it's necessary to consider an improved commutator estimate with weight owning one more higher order than before. We establish in the present paper that any weak solution belongs to the Gevrey spaces for any positive time.

\subsection{Statement of the main result}

Now we give our main result of Gevrey regularity for the spatially homogeneous Boltzmann equation for the Maxwellian molecules model, as follows:
\begin{theo}\label{main result}
  Suppose that the initial datum $f_0 \in L^1_{2+2s} \cap LlogL(\mathbb{R}^3) $. If $f \in L^\infty((0,+\infty); L_{2+2s} \cap LlogL(\mathbb{R}^3))$ is a non-negative weak solution to the Cauchy problem \eqref{BE}, then
      \begin{itemize}
      \item[i)] for the mild singularity case $0<s < \frac{1}{2}$, we have
          $f(t,\cdot) \in G^\frac{1}{2\alpha}(\mathbb{R}^3)$ for any $0<\alpha<s$ and $t>0$;

      \item[ii)] in the critical case of $s = \frac{1}{2}$, we have
          $f(t,\cdot) \in G^{s'} (\mathbb{R}^3)$ for any $s'>\frac{3}{2}$ and $t>0$;

      \item[iii)] for the strictly strong singularity case $\frac{1}{2}<s<1$, we have
          $f(t,\cdot) \in G^\frac{3}{2}(\mathbb{R}^3)$ for any $t>0$.

    \end{itemize}

\end{theo}

\begin{rema}
  From the argument in Section 4, we can claim precisely that $f(t,\cdot) \in G^{\frac{3}{2s+1}+\ep}(\mathbb{R}^3)$ for any $t>0$ and $\ep>0$ for the strong singularity case $s \in [\frac{1}{2},1)$.
\end{rema}

\subsection{The structure of the paper}

The rest of the paper is organized as follows. In the next section we give some preliminaries including some properties of the Gevrey mollifier, and commutator estimates between the collision operator and the mollifier. In Section 3 we establish the Sobolev smoothing effect for the weak solutions in some weighted Sobolev spaces, by taking advantage of the Sobolev mollifier operator. The last section is devoted to state the Gevrey regularizing effect for the weak solutions.

\section{Preliminaries}

\subsection{The mollifier operator}
To study the Gevrey regularizing effect for weak solutions of the Boltzmann equation, we consider the following exponential type mollifier (compare \cite{Lekrine-Xu-09}):
\begin{align}
  G_\delta(t,\xi)=\frac{e^{c_0 t \la \xi \ra^{2\alpha}}}
                       {1+\delta e^{c_0 t \la \xi \ra^{2\alpha}}},
\end{align}
with $\la \xi \ra=(1+|\xi|^2)^\frac{1}{2}, ~\xi \in \mathbb{R}^3$ and $c_0>0,~0<\delta<1$. It is an easy matter to check, for any $0<\delta<1$, that
\begin{align}
  G_\delta(t,\xi) \in L^\infty \left( (0,T) \times \mathbb{R}^3 \right),
\end{align}
and
\begin{align}
  \lim_{\delta \to 0} G_\delta(t,\xi) = e^{c_0 t \la \xi \ra^{2\alpha}}.
\end{align}

Denote by $G_\delta(t,D_v)$ the Fourier multiplier of symbol $G_\delta(t,\xi)$, more precisely,
\begin{align*}
  G_\delta h(t,v)= G_\delta(t,D_v) h(t,v)
  = \mathcal{F}^{-1}_{\xi \mapsto v} \left( G_\delta(t,\xi) \hat{h}(t,\xi) \right),
\end{align*}
where $\hat{h}$ represents the Fourier transform of $h$.

We aim to state the uniform bound of the term $\|G_\delta(t,D_v) f(t,v)\|_{L^2_2}$ for the weak solutions to the Cauchy problem (\ref{BE}), with respect to $\delta$. In all that follows, the same notation $G_\delta$ will stand for the pseudo-differential operators $G_\delta(t,D_v)$ or alternatively, its symbol $G_\delta(t,\xi)$, for their meanings may be inferred from the context.

We then give some properties about $G_\delta(t,\xi)$, as follows:
\begin{lemm}\label{estimates for derivations}
  Let $T>0$, then for any $t \in [0,T]$ and $\xi \in \mathbb{R}^3$, we have
  \begin{align}
    & |\partial_t \Gd | \le c_0 \la \xi \ra^{2\alpha} \Gd ,\\
    & |\partial_\xi \Gd  | \le 2\alpha c_0 t \la \xi \ra^{2\alpha-1} \Gd ,\\
    & |\partial^2_{\xi \xi} \Gd  | \le C \la \xi \ra^{2(2\alpha-1)} \Gd ,\\
    & |\partial^3_{\xi \xi \xi} \Gd  | \le C \la \xi \ra^{3(2\alpha-1)} \Gd ,\\
    & |\partial^4_{\xi \xi \xi \xi} \Gd  | \le C \la \xi \ra^{4(2\alpha-1)} \Gd,
  \end{align}
  with $C>0$ independent of $\delta$.
\end{lemm}
\begin{proof}
  By direct calculus, we can infer that
\begin{align}
  & \p_t \Gd  = c_0 \la \xi \ra^{2\alpha} \Gd  \frac{1}{1+\delta e^{c_0 t \la \xi \ra^{2\alpha}}},\\
  & \p_\xi \Gd  = 2\alpha c_0 t (1+|\xi|^2)^{\alpha-1} \xi \Gd  \frac{1}{1+\delta e^{c_0 t \la \xi \ra^{2\alpha}}},
\end{align}
\begin{align}
  \p^2_{\xi_i \xi_j} \Gd
= & \left[ 2\alpha c_0 t (1+|\xi|^2)^{\alpha-1} \right]^2 \xi_i \xi_j \Gd
      \frac{1-\delta e^{c_0 t \la \xi \ra^{2\alpha}}}{(1+\delta e^{c_0 t \la \xi \ra^{2\alpha}})^2} \\
  & + 2 \a c_0 t \left[ (1+|\xi|^2)^{\alpha-1} \delta_{ij}
                           + 2(\alpha-1) \xi_i \xi_j (1+|\xi|^2)^{\alpha-2}
                    \right] \Gd  \frac{1}{1+\delta e^{c_0 t \la \xi \ra^{2\alpha}}}, \nonumber
\end{align}
and {\small
\begin{align}
  & \p^3_{\xi_i \xi_j \xi_k} \Gd  \\
= & \left[ 2\a c_0 t (1+|\xi|^2)^{\alpha-1} \right]^3 \xi_i \xi_j \xi_k \Gd
      \frac{(1-\delta e^{c_0 t \la \xi \ra^{2\alpha}})^2}
           {(1+\delta e^{c_0 t \la \xi \ra^{2\alpha}})^3} \nonumber\\
 & + \left[ 2\a c_0 t (1+|\xi|^2)^{\alpha-1} \right]^3
      (\xi_i \delta_{ij} + \xi_j \delta_{ik} + \xi_k \delta_{ij} ) \Gd
      \frac{1-\delta e^{c_0 t \la \xi \ra^{2\alpha}}}
           {(1+\delta e^{c_0 t \la \xi \ra^{2\alpha}})^2} \nonumber \\
 & + (2\a c_0 t)^2 6 (\a -1) (1+|\xi|^2)^{2\alpha-3} \xi_i \xi_j \xi_k \Gd
    \frac{1-\delta e^{c_0 t \la \xi \ra^{2\alpha}}}
           {(1+\delta e^{c_0 t \la \xi \ra^{2\alpha}})^2} \nonumber \\
 & + 2\a c_0 t \left[ 2 (\a -1) (1+|\xi|^2)^{\alpha-2}
                       (\xi_i \delta_{ij} + \xi_j \delta_{ik} + \xi_k \delta_{ij} )
                    + 4 (\a -1) (\a -2) \xi_i \xi_j \xi_k (1+|\xi|^2)^{\alpha-3}
              \right] \Gd  \frac{1}{1+\delta e^{c_0 t \la \xi \ra^{2\alpha}}} \nonumber \\
\triangleq & A+B+C+D. \nonumber
\end{align}
}
Then the former four results of Lemma \ref{estimates for derivations} follow easily. As for the 4-th derivations in the last equality, for the sake of simplicity, we only take the term $\p_{\xi_l} A$ for example, as follows:
\begin{align}
  \p_{\xi_l} A
= & \left[ 2\a c_0 t (1+|\xi|^2)^{\alpha-1} \right]^4 \xi_i \xi_j \xi_k \xi_l \Gd
     \frac{(1-\delta e^{c_0 t \la \xi \ra^{2\alpha}})^2}
          {(1+\delta e^{c_0 t \la \xi \ra^{2\alpha}})^4} \nonumber \\
& + \left[ 2\a c_0 t (1+|\xi|^2)^{\alpha-1} \right]^3
     (\delta_{il} \xi_j \xi_k + \delta_{jl} \xi_i \xi_k + \delta_{kl} \xi_i \xi_j ) \Gd
     \frac{(1-\delta e^{c_0 t \la \xi \ra^{2\alpha}})^2}
          {(1+\delta e^{c_0 t \la \xi \ra^{2\alpha}})^3} \nonumber \\
& + (2\a c_0 t)^3 6 (\a -1) (1+|\xi|^2)^{3\alpha-4} \xi_i \xi_j \xi_k \xi_l \Gd
     \frac{(1-\delta e^{c_0 t \la \xi \ra^{2\alpha}})^3}
          {(1+\delta e^{c_0 t \la \xi \ra^{2\alpha}})^3}, \nonumber
\end{align}
combining with computations for the other three terms $\p_{\xi_l}B$, $\p_{\xi_l}C$, and $\p_{\xi_l}D$, this yields the desired result.
\end{proof}

\begin{lemm}\label{G-0 diff}
  For all $0<\delta<1$ and $\xi \in \mathbb{R}^3$, we have
  \begin{align}
    |\G (\xi) - \G (\xi^+) | \lm sin^2(\frac{\theta}{2}) \la \xi \ra^{2\a} \G (\xi^+) \G (\xi^-).
  \end{align}
\end{lemm}
\begin{proof}
  Noticing the fact $\Gd = \G (t,|\xi|)$, and denoting $s=|\xi|^2,\ s^+=|\xi^+|^2$, we have
  \begin{align}
    \G (\xi)= \widetilde{G}_\delta(s)=\frac{e^{c_0 t(1+s)^\a}}{1+\delta e^{c_0 t(1+s)^\a}}.
  \end{align}
  Moreover, we compute
    \begin{align*}
    \frac{d}{ds} \widetilde{G}_\delta (s)=\a c_0 t \widetilde{G}_\delta(s) (1+s)^{\a-1}
                                                   \frac{1}{1+\delta e^{c_0 t(1+s)^\a}} >0.
  \end{align*}

  By virtue of the Taylor formula, it holds,
  \begin{align}\label{use1}
    & |\G (\xi) - \G (\xi^+)|= |\widetilde{G}_\delta (s) - \widetilde{G}_\delta (s^+) |
   =  | (s-s^+) \int_0^1 \frac{d}{ds} \widetilde{G}_\delta (s_\tau) d\tau | \\
   \lm & |s-s^+| \int_0^1 \widetilde{G}_\delta(s_\tau) (1+s_\tau)^{\a-1} d\tau, \nonumber
  \end{align}
  where $s_\tau = (1-\tau) s^+ + \tau s$ with $\tau \in [0,1]$.

  On the other hand, the fact $|\xi^+|^2=|\xi|^2 \cos^2(\frac{\theta}{2}) $ with $\theta \in [0,\pi/2]$ implies that
  \begin{align}
    s_\tau = (1-\tau) s^+ + \tau s = (1-\tau) |\xi^+|^2 + \tau |\xi|^2 \in [\frac{1}{2}|\xi|^2,~|\xi|^2]= [\frac{s}{2},\ s],
  \end{align}
  thereby we have $(1+s_\tau)^{\a-1} \lm (1+s)^{\a-1}$ for $\a \in (0,1/2)$, and
  \begin{align}
    \widetilde{G}_\delta(s_\tau) \le \widetilde{G}_\delta(s)=\G (\xi).
  \end{align}

  Recalling the formula $|\xi|^2=|\xi^+|^2 + |\xi^-|^2$, and the following facts,
  \begin{align*}
    (1+a+b)^\a \le (1+a)^\a + (1+b)^\a,\ \ (1+\delta e^\a) (1+\delta e^\beta) \le 3 (1+\delta e^{\a+\beta}),
  \end{align*}
  we get
  \begin{align}\label{Gd inequ}
    \G (\xi) \le 3 \G (\xi^+) \G (\xi^-).
  \end{align}

  Thus we can obtain, from (\ref{use1}),
  \begin{align}
    |\G (\xi) - \G (\xi^+)| \lm \left| |\xi|^2-|\xi^+|^2 \right| \G (\xi) (1+|\xi|^2)^{\a-1}
    \lm \sin^2(\frac{\theta}{2}) \la \xi \ra^{2\a} \G (\xi^+) \G (\xi^-).
  \end{align}
\end{proof}

\subsection{Coercivity estimates for collision operator}

We introduce the following coercivity estimates for the collision operator of the Boltzmann equation (see \cite{Desvillettes2, Mori-Ukai-Xu-Yang}).
\begin{lemm}\label{Coercivity lemma}
  Assume that $g \ge 0,~g \not \equiv 0$, and further $g \in L^1_2 \cap LlogL(\mathbb{R}^3)$, then there exists a constant $C_g > 0$ depending only on $B,~\| g \|_{L^1_2}$ and $\| g \|_{L^1_2 \cap LlogL}$ such that
  \begin{align}
    \|f\|^2_{H^s} \le C_g \langle -Q(g,f),f \rangle + C \|g\|_{L^1} \|f\|^2_{L^2}
  \end{align}
  for any smooth function $f \in H^2(\mathbb{R}^3)$.
\end{lemm}

\subsection{Estimates for Commutator with weights}

In this subsection, we will give some estimates for commutators between the Boltzmann collision operator and the mollifier operator.
\begin{prop}\label{Commutator-0}
  Assume that $0<s<1$ and $0<\a \le 1/2$. For a suitable function $f$, we have,
  \begin{align}
    | \la \G Q(f,f) - Q(f,\G f), \G f \ra | \lm \| \G f \|_{L^2_2} \| \G f \|^2_{H^\a}.
  \end{align}
\end{prop}

\begin{proof}
  Thanks to the Bobylev identity and Plancherel formula, we write that
\begin{align}
  & \la \G Q(f,f) - Q(f,\G f), \G f \ra \\
= & C \Big\{ \iint \G (\xi) b(\cos \theta)
                 \left[ \hat{f}(\xi^-) \hat{f}(\xi^+) - \hat{f}(0) \hat{f}(\xi) \right]
                 \overline{(\G f)^(\xi)} d\xi d\sigma \nonumber \\
  & - \iint b(\cos \theta)
      \left[ \hat{f}(\xi^-) \G (\xi^+) \hat{f}(\xi^+) - \hat{f}(0) \G(\xi) \hat{f}(\xi) \right]
      \overline{(\G f)^(\xi)} d\xi d\sigma
  \Big\} \nonumber \\
= & C \iint b(\cos \theta) \hat{f}(\xi^-) \{\G(\xi)- \G(\xi^+)\} \hat{f}(\xi^+)
      \overline{\G(\xi) \hat{f}(\xi)}, \nonumber
\end{align}
where we have used the following notations:
\begin{align}
  \xi^-=\frac{\xi-|\xi|\sigma}{2}, \quad \xi^+=\frac{\xi+|\xi|\sigma}{2},\ \
  \cos\theta=\frac{\xi}{|\xi|} \cdot \sigma \triangleq \omega \cdot \sigma.
\end{align}

By using Lemma \ref{G-0 diff} we get
\begin{align}
  & | \la \G Q(f,f) - Q(f,\G f), \G f \ra | \\
\lm & \iint b(\cos \theta) \sin^2(\frac{\theta}{2}) |\G(\xi^-) \hat{f}(\xi^-)|\
            |\G(\xi^+) \hat{f}(\xi^+)|\ \la \xi \ra^{2\a}\ |\G(\xi) \hat{f}(\xi)|\ d\xi d\sigma
   \nonumber \\
\lm & \|\G(\xi^-) \hf(\xi^-)\|_{L^\infty}
      \int \frac{b(\cos \theta) \sin^2{\frac{\theta}{2}} }{\cos^{\a+1} \frac{\theta}{2} } d\sigma
      \left( \int \la \xi^+ \ra^{2\a} |\G(\xi^+) \hf(\xi^+)|^2 d\xi^+ \right)^\frac{1}{2}
      \|\G(\xi) \hf(\xi)\|_{H^\a} \nonumber \\
\lm & \|\G f\|_{L^1} \|\G f\|^2_{H^\a} \nonumber \\
\lm & \|\G f\|_{L^2_2} \|\G f\|^2_{H^\a}, \nonumber
\end{align}
where in the last inequality we have used the embedding $L^2_{3/2+\ep}(\mathbb{R}^3) \subset L^1(\mathbb{R}^3)$ for any $\ep>0$.
\end{proof}

\begin{prop}\label{Commutator-1}
  Assume that $1/2 \le s<1$ and $0<\a \le 1/2$. For a suitable function $f$, we have,
  \begin{align}
   & | \la v \G Q(f,f) - Q(f,v \G f),\ v \G f \ra | \\
  \lm & \Big(\|f\|_{L^1} + \|f\|_{L^1_1} + \|\G f\|_{L^2_2}\Big) \|\G f\|^2_{H^\a_1}
      + \|\G f\|_{L^2_2}\ \|\G f\|^2_{H^{(3\a-\frac{1}{2})^+}_1}.
  \nonumber
  \end{align}
\end{prop}

\begin{proof}
  Applying again the Bobylev identity and Plancherel formula, we get
  \begin{align}
      & - \la v \G Q(f,f) - Q(f,v \G f),\ v \G f \ra \\
= & C \iint b(\cos \theta)
                 \left\{ \p_\xi \Big(\G (\xi) \hat{f}(\xi^-) \hat{f}(\xi^+) \Big)
                        - \hf(\xi^-) \Big(\p_\xi(\G \hf)\Big)(\xi^+)
                 \right\}
        \overline{\p_\xi \Big(\G(\xi) \hf(\xi) \Big)} d\xi d\sigma \nonumber \\
= & C \iint b(\cos \theta) \frac{\p\xi^-}{\p\xi} \left( \p_\xi \hf \right)(\xi^-)\ \G(\xi) \hf(\xi^+)\
        \overline{\p_\xi \Big(\G(\xi) \hf(\xi) \Big)} d\xi d\sigma \nonumber \\
  & + C \iint b(\cos \theta)
         \left\{ \p_\xi \Big(\G (\xi) \hf(\xi^+) \Big) - \Big(\p_\xi(\G \hf)\Big)(\xi^+) \right\} \
        \overline{\p_\xi \Big(\G(\xi) \hf(\xi) \Big)} d\xi d\sigma \nonumber \\
= & I_1 + I_2 . \nonumber
  \end{align}

We now treat the term $I_1$. Firstly, the fact $\xi^-=\frac{\xi- |\xi| \sigma}{2}$ implies\footnote{Here we use the notation $\xi \otimes \eta=(\xi_i \eta_j)$ for two vectors $\xi=(\xi_1,\xi_2,\xi_3)$ and $\eta=(\eta_1,\eta_2,\eta_3)$.}
\begin{align}
  \frac{\p \xi^-}{\p \xi}= \frac{I-\sigma \otimes \omega }{2},
\end{align}
then we have
\begin{align}
  \left|\frac{\p \xi^-}{\p \xi}\right|= \frac{1-\sigma \cdot \omega }{2}= \sin^2{\frac{\theta}{2}}.
\end{align}

Furthermore, we can split $I_1$ into two terms, as follows:
\begin{align}
  I_1= & \iint b(\cos \theta) \frac{\p\xi^-}{\p\xi} \left( \p_\xi \hf \right)(\xi^-)\
           \G(\xi^+) \hf(\xi^+)\ \overline{\p_\xi \Big(\G(\xi) \hf(\xi) \Big)} d\xi d\sigma \\
   & + \iint b(\cos \theta) \frac{\p\xi^-}{\p\xi} \left( \p_\xi \hf \right)(\xi^-)\
          \left\{\G(\xi) - \G(\xi^+) \right\} \hf(\xi^+)\
          \overline{\p_\xi \Big(\G(\xi) \hf(\xi) \Big)} d\xi d\sigma \nonumber \\
 \triangleq & I_{11} + I_{12}, \nonumber
\end{align}
thereby we compute
\begin{align}
  |I_{11}| \lm & \iint b(\cos \theta) \sin^2 \frac{\theta}{2}\
                 \left|\left( \p_\xi \hf \right)(\xi^-)\right|\ \left|\G(\xi^+) \hf(\xi^+)\right|\
                 \left|\p_\xi \Big(\G(\xi) \hf(\xi) \Big)\right|\ d\xi d\sigma \\
  \lm & \|\p_\xi \hf\|_{L^\infty}
        \int \frac{b(\cos \theta) \sin^2 \frac{\theta}{2}}{\cos\frac{\theta}{2}}\ d\sigma
        \left( \int |\G(\xi^+) \hf(\xi^+)|^2 d\xi^+ \right)^\frac{1}{2}
        \|v \G f\|_{L^2}
  \nonumber \\
  \lm & \|f\|_{L^1_1} \|\G f\|^2_{L^2_1}, \nonumber
\end{align}
and recalling Lemma \ref{G-0 diff}, we have
\begin{align}
  |I_{12}| = & \iint b(\cos \theta) \sin^4 \frac{\theta}{2}
               \left|\G(\xi^-) \left( \p_\xi \hf \right)(\xi^-)\right|\
               \left|\G(\xi^+) \hf(\xi^+)\right|\ \la \xi \ra^{2\a}
               \left|\p_\xi \Big(\G(\xi) \hf(\xi) \Big)\right|\ d\xi d\sigma \\
  \lm & \|\G \hf\|_{L^\infty}
    \int \frac{b(\cos \theta) \sin^4 \frac{\theta}{2}}{\sin^{\a+1} \frac{\theta}{2}}\ d\sigma
         \left( \int \la \xi^- \ra^{2\a} |\G(\xi^-) \left( \p_\xi \hf \right)(\xi^-)|^2 d\xi^- \right)^\frac{1}{2}
    \|v \G f\|_{H^\a}
  \nonumber \\
  \lm & \|\G f\|_{L^1} \|\la \xi \ra^\a \G (\p_\xi \hf)\|_{L^2} \|\G f\|_{H^\a_1}
  \nonumber \\
  \lm & \|\G f\|_{L^2_2} \|\G f\|^2_{H^\a_1}.
  \nonumber
\end{align}

Note that in the last inequality we have used the assumption $0<\a \le 1/2$, which implies
\begin{align}
 & \|\la \xi \ra^\a \G (\p_\xi \hf)\|_{L^2} = \|\G (\p_\xi \hf)\|_{L^2_\a}
\lm \|\p_\xi (\G \hf)\|_{L^2_\a} + \|(\p_\xi \G) \hf\|_{L^2_\a} \nonumber \\
\lm & \|\p_\xi (\G \hf)\|_{L^2_\a} + \|\la \xi \ra^{2\a-1} \G \hf\|_{L^2_\a}
\lm \|\G f\|_{H^\a_1} + \|\G f\|_{H^\a} \nonumber \\
\lm & \|\G f\|_{H^\a_1}. \nonumber
\end{align}

Thus we get the estimate for $I_1$,
\begin{align}\label{I-1}
  |I_1| \lm  |I_{11}| + |I_{12}|
  \lm \|f\|_{L^1_1} \|\G f\|^2_{L^2_1} + \|\G f\|_{L^2_2} \|\G f\|^2_{H^\a_1}.
\end{align}

On the other hand, since
\begin{align}
  & \p_\xi \left( \G(\xi) \hf(\xi^+) \right) - \left( \p_\xi(\G \hf) \right)(\xi^+) \\
  = & \Big\{ \G(\xi) - \G(\xi^+) \Big\} \left(\p_\xi \hf \right)(\xi^+)
      + \G (\xi) \left( \frac{\p \xi^+}{\p \xi} - I \right) \left(\p_\xi \hf \right)(\xi^+)
      + \left\{ \left(\p_\xi \G\right)(\xi) - \left(\p_\xi \G\right)(\xi^+) \right\} \hf(\xi^+), \nonumber
\end{align}
we can split correspondingly $I_2$ into three terms $I_2= I_{21} + I_{22} + I_{23}$.

By virtue of Lemma \ref{G-0 diff}, we have
\begin{align}
  |I_{21}| \lm & \iint b(\cos \theta) \sin^2 \frac{\theta}{2} \left| \G(\xi^-) \hf(\xi^-) \right|
                   \left| \G(\xi^+) \left(\p_\xi \hf\right)(\xi^+) \right| \la \xi \ra^{2\a}
                   \left|\p_\xi \Big(\G(\xi) \hf(\xi) \Big)\right|\ d\xi d\sigma \\
  \lm & \|\G \hf\|_{L^\infty}
        \int \frac{\iint b(\cos \theta) \sin^2 \frac{\theta}{2}}{\cos^{\a+1} \frac{\theta}{2}} d\sigma
             \left( \int \la \xi^+ \ra^{2\a} |\G(\xi^+) \left( \p_\xi \hf \right)(\xi^+)|^2 d\xi^+ \right)^\frac{1}{2}
        \|v \G f\|_{H^\a}
  \nonumber \\
  \lm & \|\G f\|_{L^1} \|\la \xi \ra^\a \G (\p_\xi \hf)\|_{L^2} \|\G f\|_{H^\a_1}
  \nonumber \\
  \lm & \|\G f\|_{L^2_2} \|\G f\|^2_{H^\a_1}.
  \nonumber
\end{align}

Due to the fact
$$
  \left| \frac{\p \xi^+}{\p \xi} - I \right| = \left|-\frac{I-\sigma \otimes \omega}{2} \right|
= \sin^2 \frac{\theta}{2},
$$
we get
\begin{align}
  |I_{22}| \lm & \iint b(\cos \theta) \sin^2 \frac{\theta}{2} \left| \G(\xi^-) \hf(\xi^-) \right|
                   \left| \G(\xi^+) \hf(\xi^+) \right|
                   \left|\p_\xi \Big(\G(\xi) \hf(\xi) \Big)\right|\ d\xi d\sigma \\
  \lm & \|\G f\|_{L^1} \|\G f\|_{L^2} \|\G f\|_{L^2_1}
  \nonumber \\
  \lm & \|\G f\|_{L^2_2} \|\G f\|^2_{L^2_1}.
  \nonumber
\end{align}

Thanks to the Taylor expansion up to order 2, we have\footnote{For matrices $A=(a_{ij}),\ B=(b_{ij}),\ i,\ j \in \{1,2,3\}$, we agree that $A:B=(a_{ij} b_{ij})$.}
\begin{align}\label{Taylor2-1-G}
  \left(\p_\xi \G\right)(\xi) - \left(\p_\xi \G\right)(\xi^+)
= (\xi-\xi^+) \cdot \left(\p^2_{\xi \xi} \G\right)(\xi^+)
 + \int_0^1 (1-\tau) (\xi-\xi^+) \otimes (\xi-\xi^+) : \left(\p^3_{\xi \xi \xi} \G\right)(\xi_\tau) d\tau
\end{align}
with $\tau \in [0,1]$ and $\xi_\tau = (1-\tau) \xi^+ + \tau \xi$. Correspondingly, we can rewrite $I_{23}$ as follows:
\begin{align}
  I_{23} = & \iint b(\cos\theta) \hf(\xi^-) (\xi-\xi^+) \cdot \left(\p^2_{\xi \xi} \G\right)(\xi^+) \hf(\xi^+)\ \overline{\p_\xi \Big(\G(\xi) \hf(\xi) \Big)} d\xi d\sigma \\
  & + \iint \int_0^1 (1-\tau) b(\cos\theta) \hf(\xi^-)
       (\xi-\xi^+) \otimes (\xi-\xi^+) : \left(\p^3_{\xi \xi \xi} \G\right)(\xi_\tau) \hf(\xi^+)\ \overline{\p_\xi \Big(\G(\xi) \hf(\xi) \Big)} d\tau d\xi d\sigma \nonumber \\
  \triangleq & I_{231} + I_{232}. \nonumber
\end{align}

For the estimate of $I_{231}$,  we use the symmetry of cross-section $b$ with respect to $\sigma$ around the direction $\xi/|\xi|$ (see \cite{FiveGroup-regulariz, Gressman}), which forces all components of $\xi-\xi^+$ to vanish except the component in the symmetry direction. Noticing $\xi^- \perp \xi^+$, we can take the place of $\xi-\xi^+$ in $I_{231}$ by
\begin{align}
  \left \langle \xi-\xi^+, \frac{\xi}{|\xi|} \right \rangle \cdot \frac{\xi}{|\xi|}
= \left \langle \xi^-, \frac{\xi^- + \xi^+}{|\xi|} \right \rangle \cdot \frac{\xi}{|\xi|}
= \xi \ \frac{|\xi^-|^2}{|\xi|^2}
= \xi \ \sin^2 \frac{\theta}{2}.
\end{align}

Combining Lemma \ref{estimates for derivations} with the fact $4\a-1 \le 2\a$ for $\a \le 1/2$, this yields
\begin{align}
  |I_{231}| \lm & \iint b(\cos\theta) \sin^2 \frac{\theta}{2}
                 \left|\hf(\xi^-)\right| \left|\G(\xi^+) \hf(\xi^+)\right| |\xi|
                 \la \xi^+ \ra^{2(2\a-1)} \left|\p_\xi \Big(\G(\xi) \hf(\xi) \Big)\right|\ d\xi d\sigma\\
  \lm & \|\hf\|_{L^\infty}
        \int \frac{b(\cos\theta) \sin^2\frac{\theta}{2}}{\cos\frac{\theta}{2}} d\sigma\
        \|\la \xi^+ \ra^\a \G(\xi^+) \hf(\xi^+)\|_{L^2}
        \|v \G f\|_{H^\a}
  \nonumber \\
  \lm & \|f\|_{L^1}\ \|\G f\|^2_{H^\a_1}.
  \nonumber
\end{align}

Concerning the term $I_{232}$, we have
\begin{align}
  |I_{232}| \lm & \iint b(\cos\theta) \sin^2 \frac{\theta}{2}
                 \left|\G(\xi^-) \hf(\xi^-)\right| \left|\G(\xi^+) \hf(\xi^+)\right|
                 \la \xi \ra^{6\a-1} \left|\p_\xi \Big(\G(\xi) \hf(\xi) \Big)\right|\ d\xi d\sigma\\
  \lm & \|\G \hf\|_{L^\infty}
      \int\frac{b(\cos\theta) \sin^2\frac{\theta}{2}}{\cos^{3\a+\frac{1}{2}}\frac{\theta}{2}} d\sigma\
        \|\la \xi^+ \ra^{3\a-\frac{1}{2}} \G(\xi^+) \hf(\xi^+)\|_{L^2}
        \|v \G f\|_{H^{(3\a-\frac{1}{2})^+}}
  \nonumber \\
  \lm & \|\G f\|_{L^2_2}\ \|\G f\|^2_{H^{(3\a-\frac{1}{2})^+}_1}.
  \nonumber
\end{align}

Thus we obtain the estimate
\begin{align}\label{I-2}
  |I_2| \lm & |I_{21}| + |I_{22}| + |I_{231}| + |I_{232}| \\
\lm & \Big(\|f\|_{L^1} + \|\G f\|_{L^2_2}\Big) \|\G f\|^2_{H^\a_1}
    + \|\G f\|_{L^2_2}\ \|\G f\|^2_{H^{(3\a-\frac{1}{2})^+}_1}. \nonumber
\end{align}

Together with the estimates (\ref{I-1}) and (\ref{I-2}), we obtain the desired result.
\end{proof}

\begin{rema}\label{Remark-order-1}
  For $0<\a <s<1/2$, we have
  \begin{align}
    | \la v \G Q(f,f) - Q(f,v \G f),\ v \G f \ra |
  \lm \Big(\|f\|_{L^1_1} + \|\G f\|_{L^2_2}\Big) \|\G f\|^2_{H^\a_1}.
  \end{align}
\end{rema}

\begin{proof}
  We need only to revise the estimate for $I_{23}$ in the above process. The Taylor formula gives,
  \begin{align}\label{Taylor1-1-G}
    &|(\p_\xi\G)(\xi)-(\p_\xi\G)(\xi^+)|=\left|\int_0^1(\xi-\xi^+)(\p^2_{\xi\xi}\G)(\xi_\tau)d\tau\right|\\
    \lm &|\xi^-| \int_0^1 \la\xi^\tau\ra^{2(2\a-1)}\G(\xi^\tau)d\tau \nonumber \\
    \lm &\sin\frac{\theta}{2}\la\xi^\tau\ra^{(4\a-1)^+}\G(\xi^-)\G(\xi^+). \nonumber
  \end{align}

  Noticing the fact $(4\a-1)^+ \le 2\a$ for $\a \le 1/2$, hence we have
  \begin{align}
  |I_{23}| \lm & \iint b(\cos\theta) \sin\frac{\theta}{2}
                 \left|\G(\xi^-)\hf(\xi^-)\right| \left|\G(\xi^+) \hf(\xi^+)\right|
                 \la\xi\ra^{(4\a-1)^+} \left|\p_\xi \Big(\G(\xi) \hf(\xi) \Big)\right|\ d\xi d\sigma\\
  \lm & \|\G\hf\|_{L^\infty}
        \int \frac{b(\cos\theta) \sin\frac{\theta}{2}}{\cos^{\a+1}\frac{\theta}{2}} d\sigma\
        \|\la \xi^+ \ra^\a \G(\xi^+) \hf(\xi^+)\|_{L^2}
        \|v \G f\|_{H^\a}
  \nonumber \\
  \lm & \|\G f\|_{L^1}\ \|\G f\|_{H^\a}\ \|\G f\|_{H^\a_1} \nonumber \\
  \lm & \|\G f\|_{L^2_2}\ \|\G f\|^2_{H^\a_1}. \nonumber
  \end{align}

  Therefore we get
\begin{align}\label{I-2-small}
  |I_2| \lm  |I_{21}| + |I_{22}| + |I_{231}| + |I_{232}|
\lm  \|\G f\|_{L^2_2}\ \|\G f\|^2_{H^\a_1}.
\end{align}

Together with the estimates for $I_1$ (see (\ref{I-1})), this completes the proof of Remark \ref{Remark-order-1}.
\end{proof}

\begin{prop}\label{Commutator-2}
  Assume that $1/2 \le s<1$ and $0<\a \le 1/2$. For a suitable function $f$, we have,
  \begin{align}
    & | \la v \otimes v \G Q(f,f) - Q(f,v \otimes v \G f), v \otimes v \G f \ra | \\
  \lm & \|\G f\|^3_{L^2_2}
     + \Big(\|f\|_{L^1} + \|f\|_{L^1_1} + \|\G f\|_{L^2_2}\Big) \|\G f\|^2_{H^\a_2}
     + \|\G f\|_{L^2_2} \|\G f\|^2_{H^{(3\a-\frac{1}{2})^+}_2}. \nonumber
  \end{align}
\end{prop}

\begin{proof}
From the Bobylev identity and Plancherel formula, we deduce that
\begin{align}
  & \left \la v\otimes v \G Q(f,f)-Q(f,v\otimes v \G f),\ v\otimes v \G f \right \ra \\
= & C \iint b(\cos\theta) \left\{ \p^2_{\xi\xi} \left( \G(\xi) \hf(\xi^-) \hf(\xi^+) \right)
                                - \hf(\xi^-) \left(\p^2_{\xi\xi}(\G\hf)\right)(\xi^+)
                          \right\}\
      \overline{\left(\p^2_{\xi\xi}(\G\hf)\right)(\xi)}\ d\xi d\sigma \nonumber \\
=& C \iint b(\cos\theta) \left\{\left(\p^2_{\xi\xi} \hf\right)(\xi^-)\Big(\frac{\p\xi^-}{\p\xi}\Big)^2
                          + \left(\p_\xi \hf\right)(\xi^-) \left(\frac{\p^2\xi^-}{\p\xi \p\xi}\right)
                         \right\} \G(\xi)\hf(\xi^+)\
      \overline{\left(\p^2_{\xi\xi}(\G\hf)\right)(\xi)}\ d\xi d\sigma \nonumber \\
& +2C \iint b(\cos\theta) \left(\p_\xi \hf\right)(\xi^-)\ \frac{\p\xi^-}{\p\xi}\
            \p_\xi \left(\G(\xi) \hf(\xi^+)\right)
      \overline{\left(\p^2_{\xi\xi}(\G\hf)\right)(\xi)}\ d\xi d\sigma \nonumber \\
& +C \iint b(\cos\theta) \hf(\xi^-)
        \left\{\p^2_{\xi\xi} \left(\G(\xi)\hf(\xi^+) \right)-\left(\p^2_{\xi\xi}(\G\hf)\right)(\xi^+)
        \right\}\
      \overline{\left(\p^2_{\xi\xi}(\G\hf)\right)(\xi)}\ d\xi d\sigma \nonumber \\
\triangleq & II_1 + II_2 + II_3. \nonumber
\end{align}

We begin with the estimate for $II_1$, noticing the facts
\begin{align}
  \frac{\p \xi^-}{\p \xi}= \frac{I-\sigma \otimes \omega}{2}
  =\left(\frac{\delta_{ij}-\sigma_i \omega_j}{2}\right)_{3\times3},\ \ i,j \in \{1,2,3\},
\end{align}
and
\begin{align}
  \frac{\p^2 \xi^-}{\p\xi \p\xi}
  = \left(\frac{\p^2 \xi^-_{i}}{\p\xi_j \p\xi_k}\right)_{3\times3\times3}
  = \left(\frac{\sigma_i(\delta_{jk}-\omega_i \omega_k)}{2|\xi|}\right)_{3\times3\times3},
  \ \ i,j,k \in \{1,2,3\},
\end{align}
by definition of the determinant\footnote{For a $3\times3\times3$ matrix $A=(a_{ijk})$, the determinant is given by the formula $$|A|=\sum (-1)^{\tau(i_1 i_2 i_3)+\tau(j_1 j_2 j_3)+\tau(k_1 k_2 k_3)}a_{i_1 j_1 k_1}a_{i_2 j_2 k_2}a_{i_3 j_3 k_3}$$ with $\tau$ being the inversion function.}, we can deduce that
\begin{align}
  \left| \frac{\p^2 \xi^-}{\p\xi \p\xi} \right| =0.
\end{align}

Then we have
\begin{align}\label{II-1}
  |II_1| \lm & \iint b(\cos\theta) \sin^4\frac{\theta}{2}
                 \left|\G(\xi^-) \left( \p^2_{\xi\xi}\hf\right)(\xi^-)\right|\
                 \left|\G(\xi^+)\hf(\xi^+)\right|\
                 \left|\left(\p^2_{\xi\xi}(\G\hf)\right)(\xi)\right|\ d\xi d\sigma \\
  \lm & \|\G \hf\|_{L^\infty} \int b(\cos\theta) \sin^4\frac{\theta}{2} d\sigma
        \|\G (\p^2_{\xi\xi}\hf)\|_{L^2} \|\p^2_{\xi\xi}(\G\hf)\|_{L^2} \nonumber\\
  \lm & \|\G f\|_{L^1} \|\G f\|^2_{L^2_2} \nonumber \\
  \lm & \|\G f\|^3_{L^2_2}. \nonumber
\end{align}
Herein, we have used the following fact, in the last second inequality,
\begin{align}\label{G-D2-f}
  \|\G ( \p^2_{\xi\xi}\hf)\|_{L^2}
  \le & \|\p^2_{\xi\xi}(\G\hf)\|_{L^2} + \|(\p^2_{\xi\xi}\G)\hf\|_{L^2}
     + 2\|(\p_\xi\G) (\p_\xi\hf)\|_{L^2} \\
  \lm & \|\p^2_{\xi\xi}(\G\hf)\|_{L^2} + \|\la \xi \ra^{2(2\a-1)}\G\hf\|_{L^2}
      + \|\la \xi \ra^{2\a-1}\G (\p_\xi\hf)\|_{L^2} \nonumber \\
  \lm & \|\p^2_{\xi\xi}(\G\hf)\|_{L^2} + \|\G\hf\|_{L^2} + \|\G (\p_\xi\hf)\|_{L^2} \nonumber \\
  \lm & \|\p^2_{\xi\xi}(\G\hf)\|_{L^2} + \|\G\hf\|_{L^2} \nonumber \\
  \lm & \|\la v \ra^2 (\G f)\|_{L^2}. \nonumber
\end{align}

As for the term $II_2$, we rewrite it as
\begin{align}
  II_2 = & C \iint b(\cos\theta)\ \frac{\p\xi^-}{\p\xi}\left(\p_\xi \hf\right)(\xi^-)
              \left\{\G(\xi)-\G(\xi^+)\right\}\ \Big(\p_\xi \hf\Big)(\xi^+)\ \overline{\left(\p^2_{\xi\xi}(\G\hf)\right)(\xi)}\ d\xi d\sigma \nonumber \\
      & + C \iint b(\cos\theta)\ \frac{\p\xi^-}{\p\xi}\left(\p_\xi \hf\right)(\xi^-)\
             \G(\xi) \left(\p_\xi \hf\right)(\xi^+) \left(\frac{\p\xi^+}{\p\xi}-I\right)\ \overline{\left(\p^2_{\xi\xi}(\G\hf)\right)(\xi)}\ d\xi d\sigma \nonumber \\
      & + C \iint b(\cos\theta)\ \frac{\p\xi^-}{\p\xi}\left(\p_\xi \hf\right)(\xi^-)
             \left\{(\p_\xi \G)(\xi)-(\p_\xi \G)(\xi^+)\right\}\ \hf(\xi^+)\
             \overline{\left(\p^2_{\xi\xi}(\G\hf)\right)(\xi)}\ d\xi d\sigma \nonumber \\
      & + C \iint b(\cos\theta)\ \frac{\p\xi^-}{\p\xi} \left(\p_\xi \hf\right)(\xi^-)
             \left(\p_\xi(\G\hf)\right)(\xi^+)\
             \overline{\left(\p^2_{\xi\xi}(\G\hf)\right)(\xi)}\ d\xi d\sigma \nonumber \\
  \triangleq & II_{21} + II_{22} + II_{23} + II_{24} . \nonumber
\end{align}

Combining the fact $\left|\frac{\p\xi^-}{\p\xi}\right|=\sin^2\frac{\theta}{2}$ and Lemma \ref{G-0 diff}, it follows that
\begin{align}\label{II-21}
  |II_{21}| \lm & \iint b(\cos\theta)\sin^4\frac{\theta}{2}\ \la \xi \ra^{2\a}\
                  \left|\G(\xi^-) \left(\p_\xi \hf\right)(\xi^-)\right|\
                  \left|\G(\xi^+) \left(\p_\xi \hf\right)(\xi^+)\right|\
                  \left|\left(\p^2_{\xi\xi}(\G\hf)\right)(\xi)\right|\ d\xi d\sigma \\
 \lm & \int \frac{b(\cos\theta)\sin^{\frac{7}{2}}\frac{\theta}{2}}
                 {\cos^{\a+\frac{1}{2}}\frac{\theta}{2}} d\sigma \
       \|\G (\p_\xi \hf)\|_{L^4}\ \|\la \xi \ra^\a \G (\p_\xi \hf)\|_{L^4}\
       \|\la \xi \ra^\a \p^2_{\xi\xi}(\G\hf)\|_{L^2} \nonumber \\
 \lm & \|\G (\p_\xi \hf)\|_{L^4}\ \|\G (\p_\xi \hf)\|_{L^4_\a}\
       \|\G f\|_{H^\a_2}. \nonumber
\end{align}

Thanks to the Gagliardo-Nirenberg inequality\footnote{We agree that
$\Lambda f=\mathcal{F}^{-1}\left((1+|\cdot|^2)^{1/2}\hf\right)$.} (see, for instance, \cite{GuoBoLing, Nirenberg}),
\begin{align}
  \|\Lambda_\xi (\G\hf)\|_{L^4}
  \lm \|\Lambda^2_\xi(\G\hf)\|^\frac{7}{8}_{L^2}\ \|\G\hf\|^\frac{1}{8}_{L^2},
\end{align}
we obtain that
\begin{align}
  \|\G (\p_\xi \hf)\|_{L^4}
 \le & \|\p_\xi(\G\hf)\|_{L^4} + \|(\p_\xi\G)\hf\|_{L^4}
  \lm \|\p_\xi(\G\hf)\|_{L^4} + \|\la \xi \ra^{2\a-1}\|_{L^\infty} \|\G\hf\|_{L^4} \\
 \lm & \|\Lambda_\xi (\G\hf)\|_{L^4}
  \lm \|\Lambda^2_\xi(\G\hf)\|^\frac{7}{8}_{L^2}\ \|\G\hf\|^\frac{1}{8}_{L^2}
 \lm \|\G f\|_{L^2_2}, \nonumber
\end{align}
and similarly,
\begin{align}
  \|\G (\p_\xi \hf)\|_{L^4_\a} \lm \|\Lambda_\xi (\G\hf)\|_{L^4_\a}
 \lm \|\Lambda^2_\xi(\G\hf)\|^\frac{7}{8}_{L^2_\a}\ \|\G\hf\|^\frac{1}{8}_{L^2_\a}
 \lm \|\G f\|_{H^\a_2}.
\end{align}

Thus, we get the estimate for $II_{21}$:
\begin{align}
  |II_{21}| \lm \|\G f\|_{L^2_2}\ \|\G f\|^2_{H^\a_2}.
\end{align}

By virtue of the fact $\left| \frac{\p \xi^+}{\p \xi} - I \right| = \sin^2 \frac{\theta}{2}$, we have
\begin{align}\label{II-22}
  |II_{22}| \lm & \iint b(\cos\theta)\sin^4\frac{\theta}{2}\
                  \left|\G(\xi^-) \left(\p_\xi \hf\right)(\xi^-)\right|\
                  \left|\G(\xi^+) \left(\p_\xi \hf\right)(\xi^+)\right|\
                  \left|\left(\p^2_{\xi\xi}(\G\hf)\right)(\xi)\right|\ d\xi d\sigma \\
 \lm & \int \frac{b(\cos\theta)\sin^{\frac{7}{2}}\frac{\theta}{2}}
                 {\cos^{\frac{1}{2}}\frac{\theta}{2}} d\sigma \
       \|\G (\p_\xi \hf)\|_{L^4}\ \|\G (\p_\xi \hf)\|_{L^4}\
       \|\p^2_{\xi\xi}(\G\hf)\|_{L^2} \nonumber \\
 \lm & \|\G (\p_\xi \hf)\|^2_{L^4}\ \|\G f\|_{L^2_2} \nonumber \\
 \lm & \|\G f\|^3_{L^2_2}. \nonumber
\end{align}

Considering the estimate of $II_{23}$, by the Taylor formula, we have
\begin{align}
 & |(\p_\xi\G)(\xi)-(\p_\xi\G)(\xi^+)|
 = \left| \int_0^1 (\xi-\xi^+)\cdot \Big(\p^2_{\xi\xi} \G\Big)(\xi_\tau) d\tau \right| \\
 \lm & |\xi^-| \int_0^1 \la\xi_\tau \ra^{2(2\a-1)} \G(\xi_\tau) d\tau \nonumber \\
 \lm & \sin\frac{\theta}{2}\ \la\xi\ra^{(4\a-1)^+} \G(\xi^-)\G(\xi^+), \nonumber
\end{align}
where $\xi_\tau=(1-\tau)\xi^+ + \tau \xi$ with $\tau \in [0,1]$.

Observing the fact $(4\a-1)^+\le 2\a$ for $0<\a \le 1/2$, it follows that
\begin{align}\label{II-23}
  |II_{23}| \lm & \iint b(\cos\theta)\sin^3\frac{\theta}{2}\ \la \xi \ra^{(4\a-1)^+}\
                  \left|\G(\xi^-) \left(\p_\xi \hf\right)(\xi^-)\right|\
                  \left|\G(\xi^+) \hf(\xi^+)\right|\
                  \left|\left(\p^2_{\xi\xi}(\G\hf)\right)(\xi)\right|\ d\xi d\sigma \\
 \lm & \int b(\cos\theta)\sin^2\frac{\theta}{2} d\sigma \
       \|\G \left(\p_\xi \hf\right)\|_{L^2}\ \|\la \xi \ra^\a (\G hf)\|_{L^\infty}\
       \|\la \xi \ra^\a \p^2_{\xi\xi}(\G\hf)\|_{L^2} \nonumber \\
 \lm & (\|\p_\xi(\G\hf)\|_{L^2}+\|(\p_\xi\G)\hf\|_{L^2}) \|\Lambda^\a \G f\|_{L^1}\ \|\G f\|_{H^\a_2} \nonumber \\
 \lm & (\|\G f\|_{L^2_1} + \|\la \xi \ra^{2\a-1}\|_{L^\infty} \|\G\hf\|_{L^2})
       \|\Lambda^\a \G f\|_{L^2_2}\ \|\G f\|_{H^\a_2} \nonumber \\
 \lm & \|\G f\|_{L^2_2}\ \|\G f\|^2_{H^\a_2}. \nonumber
\end{align}

For the term $II_{24}$, we infer that
\begin{align}\label{II-24}
  |II_{24}| \lm & \iint b(\cos\theta)\sin^2\frac{\theta}{2}\
                    \left|\left(\p_\xi \hf\right)(\xi^-)\right|\
                    \left|\p_\xi \left(\G\hf\right)(\xi^+)\right|\
                    \left|\left(\p^2_{\xi\xi}(\G\hf)\right)(\xi)\right|\ d\xi d\sigma \\
  \lm & \|\p_\xi \hf\|_{L^\infty} \int \frac{b(\cos\theta)\sin^2\frac{\theta}{2}}{\cos\theta} d\sigma\
        \|\p_\xi(\G\hf)\|_{L^2}\ \|\G f\|_{L^2_2} \nonumber \\
  \lm & \|f\|_{L^1_1}\ \|\G f\|^2_{L^2_2}, \nonumber
\end{align}
thus, the inequalities (\ref{II-21}), (\ref{II-22}), (\ref{II-23}), and (\ref{II-24}) enable us to obtain the estimate for $II_2$:
\begin{align}\label{II-2}
  |II_2| \le |II_{21}|+|II_{22}|+|II_{23}|+|II_{24}|
  \lm (\|f\|_{L^1_1} + \|\G f\|_{L^2_2})\ \|\G f\|^2_{H^\a_2}.
\end{align}

Now we deal with the term $II_3$, firstly we write that
\begin{align}
  II_3 = & C \iint b(\cos\theta) \hf(\xi^-)
       \left\{\p^2_{\xi\xi}\Big(\G(\xi)\hf(\xi^+)\Big) - \Big(\p^2_{\xi\xi}(\G\hf)\Big)(\xi^+)\right\}
       \overline{\left(\p^2_{\xi\xi}(\G\hf)\right)(\xi)}\ d\xi d\sigma \nonumber \\
  = & C \iint b(\cos\theta) \hf(\xi^-) \Big\{(\p^2_{\xi\xi}\G)(\xi)-(\p^2_{\xi\xi}\G)(\xi^+)\Big\}\
              \hf(\xi^+)\ \overline{\left(\p^2_{\xi\xi}(\G\hf)\right)(\xi)}\ d\xi d\sigma \nonumber \\
  &+C \iint b(\cos\theta) \hf(\xi^-)\ \G(\xi) \Big(\p_\xi \hf\Big)(\xi^+)
            \left(\frac{\p^2\xi^+}{\p\xi \p\xi}\right)\
            \overline{\left(\p^2_{\xi\xi}(\G\hf)\right)(\xi)}\ d\xi d\sigma \nonumber \\
  &+C \iint b(\cos\theta) \hf(\xi^-)
            \left\{(\p_\xi\G)(\xi) \left(\frac{\p\xi^+}{\p\xi}\right) - (\p_\xi\G)(\xi^+)\right\}
            \Big(\p_\xi \hf\Big)(\xi^+)\
            \overline{\left(\p^2_{\xi\xi}(\G\hf)\right)(\xi)}\ d\xi d\sigma \nonumber \\
  &+C \iint b(\cos\theta) \hf(\xi^-)
            \left\{\G(\xi) \Big(\frac{\p\xi^+}{\p\xi}\Big)^2 - \G(\xi^+)\right\}
            \left(\p^2_{\xi\xi}\hf\right)(\xi^+)\
            \overline{\left(\p^2_{\xi\xi}(\G\hf)\right)(\xi)}\ d\xi d\sigma \nonumber \\
\triangleq & II_{31} + \Psi + II_{32} + II_{33}. \nonumber
\end{align}

We then turn to the term $II_{31}$. The Taylor formula up to order 2 gives that
\begin{align}
  (\p^2_{\xi\xi}\G)(\xi) - (\p^2_{\xi\xi}\G)(\xi^+)
 = (\xi-\xi^+) \cdot (\p^3_{\xi\xi\xi}\G)(\xi^+)
  + \int_0^1 (1-\tau) (\xi-\xi^+) \otimes (\xi-\xi^+):(\p^4_{\xi\xi\xi\xi}\G)(\xi^\tau)\ d\tau
\end{align}
with $\tau \in [0,1]$ and $\xi_\tau= (1-\tau)\xi^+ +\tau \xi$. Then we can decompose $II_{31}$ into two corresponding terms $II_{31}=II_{311} + II_{312}$.

By the symmetry of $b$ with respect to $\sigma$ mentioned before, we can take the place of $\xi-\xi^+$ in $II_{311}$ by
\begin{align}
  \left \langle \xi-\xi^+, \frac{\xi}{|\xi|} \right \rangle \cdot \frac{\xi}{|\xi|}
= \xi \ \sin^2 \frac{\theta}{2},
\end{align}
then it follows that
\begin{align}
  |II_{311}| \lm &\iint b(\cos\theta)\sin^2\frac{\theta}{2}\ \left|\hf(\xi^-)\right|\
                    |\xi|\la\xi^+\ra^{6\a-3} \left|\G(\xi^+)\hf(\xi^+)\right|\
                    \left|\left(\p^2_{\xi\xi}(\G\hf)\right)(\xi)\right|\ d\xi d\sigma \\
  \lm & \|\hf\|_{L^\infty}
        \int \frac{b(\cos\theta)\sin^2\frac{\theta}{2}}{\cos\frac{\theta}{2}}\ d\sigma\
        \|\la \cdot \ra^{3\a-1}\G\hf\|_{L^2}\ \|\la\xi\ra^{3\a-1} \p^2_{\xi\xi}(\G\hf)\|_{L^2}
        \nonumber \\
  \lm & \|f\|_{L^1}\ \|\G f\|_{L^2}\ \|\G f\|_{L^2_2},\nonumber
\end{align}
due to the assumption $\a \le 1/2$.

Furthermore, by Lemma \ref{estimates for derivations}, we can derive that
\begin{align}
  |II_{312}| \lm &\iint b(\cos\theta)\sin^2\frac{\theta}{2}\ \left|\G(\xi^-)\hf(\xi^-)\right|\
                    \la\xi^+\ra^{8\a-2} \left|\G(\xi^+)\hf(\xi^+)\right|\
                    \left|\left(\p^2_{\xi\xi}(\G\hf)\right)(\xi)\right|\ d\xi d\sigma \\
  \lm & \|\G\hf\|_{L^\infty}
        \int \frac{b(\cos\theta)\sin^2\frac{\theta}{2}}{\cos^{4\a}\frac{\theta}{2}}\ d\sigma\
        \|\la \cdot \ra^{4\a-1}\G\hf\|_{L^2}\ \|\la\xi\ra^{4\a-1} \p^2_{\xi\xi}(\G\hf)\|_{L^2}
        \nonumber \\
  \lm & \|\G f\|_{L^1}\ \|\G f\|_{H^{(4\a-1)^+}}\ \|\G f\|_{H^{(4\a-1)^+}_2} \nonumber \\
  \lm & \|\G f\|_{L^2_2}\ \|\G f\|^2_{H^{(4\a-1)^+}_2}. \nonumber
\end{align}

Combining the above two inequalities gives that
\begin{align}\label{II-31}
  |II_{31}| \lm \|f\|_{L^1}\ \|\G f\|^2_{L^2_2} + \|\G f\|_{L^2_2}\ \|\G f\|^2_{H^{(4\a-1)^+}_2}.
\end{align}

As for the term $II_{32}$, we rewrite it as
\begin{align}
  II_{32}=&2C \iint b(\cos\theta)\hf(\xi^-)\left\{(\p_\xi\G)(\xi)-(\p_\xi\G)(\xi^+)\right\}
                \Big(\p_\xi\hf\Big)(\xi^+)\ \frac{\p\xi^+}{\p\xi}\
                \overline{\left(\p^2_{\xi\xi}(\G\hf)\right)(\xi)}\ d\xi d\sigma \\
  &+2C \iint b(\cos\theta)\hf(\xi^-) \Big(\p_\xi\G\Big)(\xi^+)\
                \Big(\p_\xi\hf\Big)(\xi^+)\ \left(\frac{\p\xi^+}{\p\xi}-I\right)
                \overline{\left(\p^2_{\xi\xi}(\G\hf)\right)(\xi)}\ d\xi d\sigma \nonumber\\
\triangleq & II_{321} + II_{322}. \nonumber
\end{align}

Thanks to the Taylor formula (\ref{Taylor2-1-G}), we can split $II_{321}$ into two terms $II_{321}= II_{3211}+II_{3212}$, correspondingly. Following along the same lines of that of treating $II_{31}$, we have firstly
\begin{align}
  |II_{3211}| \lm &\iint b(\cos\theta)\sin^2\frac{\theta}{2}\cos^2\frac{\theta}{2}\
                    \left|\hf(\xi^-)\right|\ |\xi|\la\xi^+\ra^{4\a-2}
                    \left|\G(\xi^+)\Big(\p_\xi\hf\Big)(\xi^+)\right|\
                    \left|\left(\p^2_{\xi\xi}(\G\hf)\right)(\xi)\right|\ d\xi d\sigma \\
  \lm & \|\hf\|_{L^\infty}
        \int \frac{b(\cos\theta)\sin^2\frac{\theta}{2}}{\cos\frac{\theta}{2}}\ d\sigma\
        \|\G(\p_\xi\hf)\|_{L^2_\a}\ \|\la\xi\ra^{\a} \p^2_{\xi\xi}(\G\hf)\|_{L^2}
        \nonumber \\
  \lm & \|f\|_{L^1}\ \|\G f\|_{H^\a_1}\ \|\G f\|_{H^\a_2},\nonumber
\end{align}
where we have used the facts $4\a-1 \le 2\a$ and $\|\G(\p_\xi\hf)\|_{L^2} \lm \|\G f\|_{L^2_1}$ for $\a \le 1/2$.

Secondly, we have
\begin{align}
  |II_{3212}| \lm &\iint b(\cos\theta)\sin^2\frac{\theta}{2}\cos^2\frac{\theta}{2}\
                    \left|\G(\xi^-)\hf(\xi^-)\right|\ |\xi|^2 \la\xi^+\ra^{6\a-3}
                    \left|\G(\xi^+)\Big(\p_\xi\hf\Big)(\xi^+)\right|\
                    \left|\left(\p^2_{\xi\xi}(\G\hf)\right)(\xi)\right|\ d\xi d\sigma \\
  \lm & \|\G\hf\|_{L^\infty}
        \int \frac{b(\cos\theta)\sin^2\frac{\theta}{2}}{\cos^{(3\a-3/2)}\frac{\theta}{2}}\ d\sigma\
        \|\G(\p_\xi\hf)\|_{L^2_{(3\a-\frac{1}{2})^+}}\
        \|\p^2_{\xi\xi}(\G\hf)\|_{H^{(3\a-\frac{1}{2})^+}}
        \nonumber \\
  \lm & \|\G f\|_{L^2_2}\ \|\G f\|^2_{H^{(3\a-\frac{1}{2})^+}_2}. \nonumber
\end{align}

On the other hand, by the fact $\left|\frac{\p\xi^+}{\p\xi}-I\right|=\sin^2\frac{\theta}{2}$, we get
\begin{align}
  |II_{322}| \lm &\iint b(\cos\theta)\sin^2\frac{\theta}{2}\
                    \left|\hf(\xi^-)\right|\ \la\xi^+\ra^{2\a-1}
                    \left|\G(\xi^+)\Big(\p_\xi\hf\Big)(\xi^+)\right|\
                    \left|\left(\p^2_{\xi\xi}(\G\hf)\right)(\xi)\right|\ d\xi d\sigma \\
  \lm & \|\hf\|_{L^\infty}
        \int \frac{b(\cos\theta)\sin^2\frac{\theta}{2}}{\cos\frac{\theta}{2}}\ d\sigma\
        \|\G(\p_\xi\hf)\|_{L^2}\ \|\p^2_{\xi\xi}(\G\hf)\|_{L^2}
        \nonumber \\
  \lm & \|f\|_{L^1}\ \|\G f\|^2_{L^2_2}.\nonumber
\end{align}

Together with the three above inequalities, we obtain
\begin{align}\label{II-32}
  |II_{32}| \le |II_{3211}| +|II_{3212}| +|II_{322}|
 \lm \|f\|_{L^1}\ \|\G f\|^2_{H^\a_2} + \|\G f\|_{L^2_2}\ \|\G f\|^2_{H^{(3\a-\frac{1}{2})^+}_2}.
\end{align}

Considering the last term $II_{33}$, we have
\begin{align}
  II_{33}= &C \iint b(\cos\theta) \hf(\xi^-)
            \left\{\G(\xi) - \G(\xi^+)\right\} \Big(\frac{\p\xi^+}{\p\xi}\Big)^2
            \left(\p^2_{\xi\xi}\hf\right)(\xi^+)\
            \overline{\left(\p^2_{\xi\xi}(\G\hf)\right)(\xi)}\ d\xi d\sigma \nonumber \\
 &+C \iint b(\cos\theta) \hf(\xi^-)\ \G(\xi^+) \left(\p^2_{\xi\xi}\hf\right)(\xi^+)\
              \left\{\Big(\frac{\p\xi^+}{\p\xi}\Big)^2-I \right\}
              \overline{\left(\p^2_{\xi\xi}(\G\hf)\right)(\xi)}\ d\xi d\sigma \nonumber \\
 \triangleq & II_{331} + II_{332}. \nonumber
\end{align}

Then, Lemma \ref{G-0 diff} gives that
\begin{align}
  |II_{331}| \lm &\iint b(\cos\theta)\sin^2\frac{\theta}{2}\cos^4\frac{\theta}{2}\
                   \left|\G(\xi^-)\hf(\xi^-)\right|\ \la\xi\ra^{2\a}
                   \left|\G(\xi^+)\left(\p^2_{\xi\xi}\hf\right)(\xi^+)\right|\
                   \overline{\left(\p^2_{\xi\xi}(\G\hf)\right)(\xi)}\ d\xi d\sigma \\
 \lm &\|\G \hf\|_{L^\infty}
      \int b(\cos\theta)\sin^2\frac{\theta}{2}\cos^{3-\a}\frac{\theta}{2}\ d\sigma \
      \|\G (\p^2_{\xi\xi}\hf)\|_{L^2_\a}\ \|\G f\|_{H^\a_2} \nonumber \\
 \lm &\|\G f\|_{L^2_2}\ \|\G f\|^2_{H^\a_2}, \nonumber
\end{align}
where we have used the estimate $\|\G (\p^2_{\xi\xi}\hf)\|_{L^2_\a} \lm \|\G f\|^2_{H^\a_2}$ due to (\ref{G-D2-f}).

As for the term $II_{332}$, since
\begin{align}
  \left|\Big(\frac{\p\xi^+}{\p\xi}\Big)^2-I \right|
 =\left|\Big(\frac{\p\xi^+}{\p\xi}-I\Big) \Big(\frac{\p\xi^+}{\p\xi}-I\Big) \right|
 \le C \sin^2\frac{\theta}{2},
\end{align}
then we obtain
\begin{align}
  |II_{332}| \lm &\iint b(\cos\theta)\sin^2\frac{\theta}{2}\ \left|\hf(\xi^-)\right|\
                   \left|\G(\xi^+)\left(\p^2_{\xi\xi}\hf\right)(\xi^+)\right|\
                   \overline{\left(\p^2_{\xi\xi}(\G\hf)\right)(\xi)}\ d\xi d\sigma \\
 \lm &\|\hf\|_{L^\infty}
      \int \frac{b(\cos\theta)\sin^2\frac{\theta}{2}}{\cos\frac{\theta}{2}} d\sigma \
      \|\G (\p^2_{\xi\xi}\hf)\|_{L^2}\ \|\G f\|_{L^2_2} \nonumber \\
 \lm &\|f\|_{L^1}\ \|\G f\|^2_{L^2_2}. \nonumber
\end{align}

Thereby, we get the estimate for $II_{33}$,
\begin{align}\label{II-33}
  |II_{33}| \le |II_{331}| + |II_{332}| \lm (\|f\|_{L^1} + \|\G f\|_{L^2_2})\ \|\G f\|^2_{H^\a_2}.
\end{align}

Together with the estimates (\ref{II-31}), (\ref{II-32}), and (\ref{II-33}), and observing the fact
$ \left| \frac{\p^2 \xi^-}{\p\xi \p\xi} \right| =0 $ implies that
\begin{align}
  |\Psi|=0,
\end{align}
we can conclude the estimate of $II_3$:
\begin{align}\label{II-3}
  |II_3| \le & |II_{31}| +|\Psi| + |II_{32}| + |II_{33}| \\
 \lm & \Big(\|f\|_{L^1} + \|\G f\|_{L^2_2}\Big) \|\G f\|^2_{H^\a_2}
     + \|\G f\|_{L^2_2} \|\G f\|^2_{H^{(3\a-\frac{1}{2})^+}_2}.\nonumber
\end{align}

Combining this inequality with (\ref{II-1}), (\ref{II-2}) completes the whole proof.
\end{proof}

\begin{rema}\label{Remark-order-2}
 For $0<\a <s< 1/2$, we have,
  \begin{align}
    | \la v \otimes v \G Q(f,f) - Q(f,v \otimes v \G f), v \otimes v \G f \ra |
  \lm \Big(\|f\|_{L^1} + \|f\|_{L^1_1} + \|\G f\|_{L^2_2}\Big) \|\G f\|^2_{H^\a_2}.
  \end{align}
\end{rema}

\begin{proof}
  It suffices to revise the estimates for $II_{31}$ and $II_{321}$, relying on Taylor expansion of order 1.

Because
\begin{align}
 & |(\p^2_{\xi\xi}\G)(\xi) - (\p^2_{\xi\xi}\G)(\xi^+)|
 = |\int_0^1 (\xi-\xi^+) \cdot (\p^3_{\xi\xi\xi}\G)(\xi_\tau)d\tau| \nonumber \\
 \lm& |\xi^-| \int_0^1 \la \xi_\tau \ra^{3(2\a-1)} \G(\xi_\tau) d\tau
 \lm \sin\frac{\theta}{2} \la \xi \ra^{(6\a-2)^+}\G(\xi^-)\G(\xi^+), \nonumber
\end{align}
using the fact $(6\a-2)^+ \le 2\a$, we can derive that
\begin{align}
  |II_{31}| \lm &\iint b(\cos\theta)\sin\frac{\theta}{2}\ \left|\G(\xi^-)\hf(\xi^-)\right|\
                    \la\xi^+\ra^{(6\a-2)^+} \left|\G(\xi^+)\hf(\xi^+)\right|\
                    \left|\left(\p^2_{\xi\xi}(\G\hf)\right)(\xi)\right|\ d\xi d\sigma \\
  \lm & \|\G\hf\|_{L^\infty}
        \int \frac{b(\cos\theta)\sin^2\frac{\theta}{2}}{\cos^{\a+1}\frac{\theta}{2}}\ d\sigma\
        \|\la \xi \ra^\a \G\hf\|_{L^2}\ \|\la\xi\ra^\a \p^2_{\xi\xi}(\G\hf)\|_{L^2}
        \nonumber \\
  \lm & \|\G f\|_{L^2_2}\ \|\G f\|^2_{H^\a_2}. \nonumber
\end{align}

For $II_{321}$, applying the Taylor expansion (\ref{Taylor1-1-G}) and the inequality (\ref{G-D2-f}) ensures that
\begin{align}
  |II_{321}| \lm &\iint b(\cos\theta)\sin\frac{\theta}{2}\cos^2\frac{\theta}{2}\
                    \left|\G\hf(\xi^-)\right|\ \la\xi^+\ra^{(4\a-1)^+}
                    \left|\G(\xi^+)\Big(\p_\xi\hf\Big)(\xi^+)\right|\
                    \left|\left(\p^2_{\xi\xi}(\G\hf)\right)(\xi)\right|\ d\xi d\sigma \\
  \lm & \|\G\hf\|_{L^\infty}
        \int \frac{b(\cos\theta)\sin\frac{\theta}{2}}{\cos^{\a-1}\frac{\theta}{2}}\ d\sigma\
        \|\G(\p_\xi\hf)\|_{L^2_\a}\ \|\la\xi\ra^{\a} \p^2_{\xi\xi}(\G\hf)\|_{L^2}
        \nonumber \\
  \lm & \|\G f\|_{L^1}\ \|\G f\|_{H^\a_1}\ \|\G f\|_{H^\a_2}. \nonumber
\end{align}

Finally, we can obtain the desired result:
\begin{align}
  | \la v \otimes v \G Q(f,f) - Q(f,v \otimes v \G f), v \otimes v \G f \ra |
 \lm \Big(\|f\|_{L^1} + \|f\|_{L^1_1} + \|\G f\|_{L^2_2}\Big) \|\G f\|^2_{H^\a_2}.
\end{align}
\end{proof}

\section{Sobolev regualrity for weak solutions}

In this section, we will study the regularizing effect of the weak solutions for the Cauchy problem of the Boltzmann equation in some weighted Sobolev spaces.

\begin{theo}\label{Thm Sobolev}
  Assume that the initial datum $f_0 \in L^1_{2+2s} \cap LlogL(\mathbb{R}^3)$. Let $f \in L^\infty((0,\ +\infty);\ L^1_2 \cap LlogL(\mathbb{R}^3))$ be a non-negative weak solution of the Cauchy problem of the Boltzmann equation (\ref{BE}), then $$f(t,\cdot) \in H^{+\infty}_2(\mathbb{R}^3)$$ for any $t>0$.
\end{theo}

We introduce the following mollifier to help us prove the regularity of weak solutions in Sobolev spaces,
\begin{align}
    M_\delta(t,\xi)=\frac{\la \xi \ra^{Nt-2}}{(1+\delta |\xi|^2)^{N_0}}
\end{align}
for $0<\delta<1$ and $2N_0=NT_0+3,\ t\in [0,T_0]$. Then we agree that, in what follows, the notation $M_\delta(t,D_v)$ stands for the Fourier multiplier of symbol $M_\delta(t,\xi)$, that is to say,
\begin{align*}
  M_\delta h(t,v)= M_\delta(t,D_v) h(t,v)
  = \mathcal{F}^{-1}_{\xi \mapsto v} \left( M_\delta(t,\xi) \hat{h}(t,\xi) \right).
\end{align*}

We give some properties about $M_\delta(t,\xi)$ at first:
\begin{lemm}\label{estimates for derivations-M}
  Let $T>0$, then for any $t \in [0,T]$ and $\xi \in \mathbb{R}^3$, we have
  \begin{align}
    & |\partial_t \Md| \le N \log(\la\xi\ra) \Md ,\\
    & |\partial^{k}_\xi \Md| \le C_k \la \xi \ra^{-k} \Md ,\\
    & |\Md| \le C \M(t,\xi^+), \label{Md inequ}
  \end{align}
  with $C,\ C_k>0$ independent of $\delta$.
\end{lemm}

\begin{proof}
  A direct calculation gives that
  \begin{align*}
    & \log \Md =\frac{Nt-2}{2}\log(1+|\xi|^2)-N_0 \log(1+\delta|\xi|^2),\\
    & \p_t \M = \frac{N}{2}\log(1+|\xi|^2)\Md,
  \end{align*}
  which yields the first result. The left two results are easy to check, and thus omitted here.
\end{proof}

\subsection{Commutator estimates with Sobolev mollifier}

We will estimate the commutator between the collision operator and the Sobolev mollifier operator, as follows:
\begin{prop}\label{Commutator-0-M}
  Suppose that $0<s<1$. For a suitable function $f$, we have,
  \begin{align}
    |\la \M Q(f,f) - Q(f,\M f), \M f \ra| \lm \|f\|_{L^1} \|\M f\|^2_{L^2}.
  \end{align}
\end{prop}

\begin{proof}
  We can write that,
\begin{align}
  & \la \M Q(f,f) - Q(f,\M f), \M f \ra \\
= & C \Big\{ \iint \M (\xi) b(\cos \theta)
                 \left[ \hat{f}(\xi^-) \hat{f}(\xi^+) - \hat{f}(0) \hat{f}(\xi) \right]
                 \overline{(\M f)^(\xi)} d\xi d\sigma \nonumber \\
  & - \iint b(\cos \theta)
      \left[ \hat{f}(\xi^-) \M (\xi^+) \hat{f}(\xi^+) - \hat{f}(0) \M(\xi) \hat{f}(\xi) \right]
      \overline{(\M f)^(\xi)} d\xi d\sigma
  \Big\} \nonumber \\
= & C \iint b(\cos \theta) \hat{f}(\xi^-) \{\M(\xi)- \M(\xi^+)\} \hat{f}(\xi^+)
      \overline{\M(\xi) \hat{f}(\xi)} d\xi d\sigma. \nonumber
\end{align}

From the following Taylor expansion up to order 2,
\begin{align}\label{Taylor2-0-M}
  \M(\xi)-\M(\xi^+)=(\xi-\xi^+)(\p_\xi\M)(\xi^+) + \int_0^1 (1-\tau)(\xi-\xi^+)\otimes(\xi-\xi^+):(\p^2_{\xi\xi}\M)(\xi_\tau)\ d\tau,
\end{align}
with $\xi_\tau=(1-\tau)\xi^+ + \tau \xi$ for $\tau \in [0,1]$, we get
\begin{align}
  & \la \M Q(f,f) - Q(f,\M f), \M f \ra \\
= & C \iint b(\cos \theta) \hf(\xi^-) (\xi-\xi^+)(\p_\xi\M)(\xi^+) \hf(\xi^+)
      \overline{\M(\xi) \hf(\xi)} d\xi d\sigma\nonumber \\
& +C \iiint_0^1 (1-\tau)b(\cos \theta) \hf(\xi^-) (\xi-\xi^+)\otimes(\xi-\xi^+):(\p^2_{\xi\xi}\M)(\xi_\tau)\ \hf(\xi^+)
      \overline{\M(\xi) \hf(\xi)}d\tau d\xi d\sigma \nonumber \\
\triangleq& A_1 + A_2. \nonumber
\end{align}

By the symmetry property of $b$ with respect to $\sigma$, we can substitute $\sin^2\frac{\theta}{2}\xi$ for $\xi-\xi^+$ in $A_1$, and get
\begin{align}
  |A_1| \le &\left|\iint b(\cos\theta)\sin^2\th2\ |\hf(\xi^-)|\ |\xi|\ |\p_\xi\M(\xi^+)|\
                         |\hf(\xi^+)|\ |\M(\xi)\hf(\xi)|\ d\xi d\sigma
             \right| \\\no
  \lm &\iint b(\cos\theta)\sin^2\th2\ |\hf(\xi^-)|\ |\xi|\la\xi^+\ra^{-1}\ |\M(\xi^+)\hf(\xi^+)|\
             |\M(\xi)\hf(\xi)|\ d\xi d\sigma \\\no
  \lm &\|\hf\|_{L^\infty} \int \frac{b(\cos\theta)\sin^2\th2}{\cos\th2} d\sigma\
       \|\M \hf\|^2_{L^2} \\\no
  \lm &\|f\|_{L^1}\ \|\M \hf\|^2_{L^2}.
\end{align}

As for the term $A_2$, we have
\begin{align}
  |A_2| \le &\left|\iiint_0^1 (1-\tau) b(\cos\theta) |\xi^-|^2 |\hf(\xi^-)|\
                     |\p^2_{\xi\xi}\M(\xi^+)|\ |\hf(\xi^+)|\ |\M(\xi)\hf(\xi)|\ \tau d\xi d\sigma
             \right| \\\no
  \lm &\iint b(\cos\theta)\sin^2\th2\ |\hf(\xi^-)|\ |\M(\xi^+)|\ |\hf(\xi^+)|\ |\M(\xi)\hf(\xi)|\
             d\xi d\sigma \\\no
  \lm &|\hf(\xi^-)|_{L^\infty}\ \int \frac{b(\cos\theta)\sin^2\th2}{\cos\th2} d\sigma\
       \|\M \hf\|^2_{L^2} \\\no
  \lm &\|f\|_{L^1}\ \|\M \hf\|^2_{L^2}.
\end{align}

Therefore, we can obtain the needed result from the above two inequalities.
\end{proof}

\begin{prop}\label{Commutator-1-M}
  Suppose that $0<s<1$. For a suitable function $f$, we have,
  \begin{align}
    | \la v\M Q(f,f) - Q(f,v\M f), v\M f \ra |
    \lm \Big(\|f\|_{L^1}+\|f\|_{L^1_1}\Big) \|\M f\|^2_{L^2_1}.
  \end{align}
\end{prop}

\begin{proof}
  Arguing as Proposition \ref{Commutator-1}, we write that
  \begin{align}
      & - \la v \M Q(f,f) - Q(f,v \M f),\ v \M f \ra \\
= & C \iint b(\cos \theta)
                 \left\{ \p_\xi \Big(\M (\xi) \hat{f}(\xi^-) \hat{f}(\xi^+) \Big)
                        - \hf(\xi^-) \Big(\p_\xi(\M \hf)\Big)(\xi^+)
                 \right\}
        \overline{\p_\xi \Big(\M(\xi) \hf(\xi) \Big)} d\xi d\sigma \nonumber \\
= & C \iint b(\cos \theta) \frac{\p\xi^-}{\p\xi} \left( \p_\xi \hf \right)(\xi^-)\ \M(\xi) \hf(\xi^+)\
        \overline{\p_\xi \Big(\M(\xi) \hf(\xi) \Big)} d\xi d\sigma \nonumber \\
  & + C \iint b(\cos \theta)
         \left\{ \p_\xi \Big(\M (\xi) \hf(\xi^+) \Big) - \Big(\p_\xi(\M \hf)\Big)(\xi^+) \right\} \
        \overline{\p_\xi \Big(\M(\xi) \hf(\xi) \Big)} d\xi d\sigma \nonumber \\
= & I_1 + I_2 . \nonumber
  \end{align}

Firstly, the fact $\M(\xi) \lm \M(\xi^+)$ gives that
\begin{align}\label{I-1-M}
  |I_{1}|
  \lm & \iint b(\cos \theta) \sin^2 \frac{\theta}{2}\
                 \left|\left( \p_\xi \hf \right)(\xi^-)\right|\ \left|\M(\xi^+) \hf(\xi^+)\right|\
                 \left|\p_\xi \Big(\M(\xi) \hf(\xi) \Big)\right|\ d\xi d\sigma \\\no
  \lm & \|\p_\xi \hf\|_{L^\infty}
        \int \frac{b(\cos \theta) \sin^2 \frac{\theta}{2}}{\cos\frac{\theta}{2}}\ d\sigma
        \left( \int |\M(\xi^+) \hf(\xi^+)|^2 d\xi^+ \right)^\frac{1}{2}
        \|v \M f\|_{L^2} \\\no
  \lm & \|f\|_{L^1_1} \|\M f\|^2_{L^2_1}.
\end{align}

Furthermore, observe the fact
\begin{align}
  & \p_\xi \left( \M(\xi) \hf(\xi^+) \right) - \left( \p_\xi(\M \hf) \right)(\xi^+) \\
  = & \Big\{ \M(\xi) - \M(\xi^+) \Big\} \left(\p_\xi \hf \right)(\xi^+)
      + \M (\xi) \left( \frac{\p \xi^+}{\p \xi} - I \right) \left(\p_\xi \hf \right)(\xi^+)
      + \left\{ \left(\p_\xi \M\right)(\xi) - \left(\p_\xi \M\right)(\xi^+) \right\} \hf(\xi^+), \nonumber
\end{align}
then correspondingly, the term $I_2$ can be reformulated as $I_2= I_{21} + I_{22} + I_{23}$.

The Taylor expansion (\ref{Taylor2-0-M}) yields that
\begin{align}
  I_{21} =& \iint b(\cos \theta) \hf(\xi^-)\ (\xi-\xi^+)(\p_\xi\M)(\xi^+)\ \Big(\p_\xi\hf\Big)(\xi^+)\
                 \overline{\p_\xi \Big(\M(\xi) \hf(\xi) \Big)} d\xi d\sigma \\\no
 &+ \iiint_0^1 (1-\tau) b(\cos \theta) \hf(\xi^-)\
               (\xi-\xi^+)\otimes(\xi-\xi^+):(\p^2_{\xi\xi}\M)(\xi_\tau)\ \Big(\p_\xi\hf\Big)(\xi^+)\
               \overline{\p_\xi \Big(\M(\xi) \hf(\xi) \Big)} d\xi d\sigma \\\no
 \triangleq & I_{211} + I_{212}.
\end{align}

The symmetry property enables us to take the place of $\xi-\xi^+$ by $\sin^2\th2 \cdot \xi$ in $I_{211}$, thereby we get
\begin{align}
  |I_{211}| \lm & \iint b(\cos \theta) \sin^2 \frac{\theta}{2} \left|\hf(\xi^-) \right|
                   |\xi|\la\xi^+\ra^{-1}
                   \left| \M(\xi^+) \left(\p_\xi \hf\right)(\xi^+) \right|
                   \left|\p_\xi \Big(\M(\xi) \hf(\xi) \Big)\right|\ d\xi d\sigma \\\no
  \lm & \|\hf\|_{L^\infty}
        \int \frac{\iint b(\cos \theta) \sin^2 \frac{\theta}{2}}{\cos \frac{\theta}{2}} d\sigma\
             \|\M (\p_\xi \hf)\|_{L^2} \|v \M f\|_{L^2} \\\no
  \lm & \|f\|_{L^1} \|\M f\|^2_{L^2_1},
\end{align}
where we have used the estimate
\begin{align}
  &\|\M (\p_\xi\hf)\|_{L^2}
 \le \|\p_\xi (\M\hf)\|_{L^2} + \|(\p_\xi\M)\hf\|_{L^2} \\\no
 \le &\|\p_\xi (\M\hf)\|_{L^2} + \|\la\xi\ra^{-1}\|_{L^\infty}\ \|\M\hf\|_{L^2}
 \le \|\p_\xi (\M\hf)\|_{L^2}.
\end{align}

As for $I_{212}$, we have
\begin{align}
  |I_{212}| \le& \iiint_0^1 (1-\tau) b(\cos \theta) \sin^2\th2\ |\hf(\xi^-)|\
                    |\xi|^2\la\xi_\tau\ra^{-2}\ |\M(\xi_\tau)|
                    \left|\Big(\p_\xi\hf\Big)(\xi^+)\right|\
                    \left|\p_\xi \Big(\M(\xi) \hf(\xi) \Big)\right| d\xi d\sigma \\\no
  \lm & \|\hf\|_{L^\infty}
        \int \frac{\iint b(\cos \theta) \sin^2 \frac{\theta}{2}}{\cos \frac{\theta}{2}} d\sigma\
             \|\M ( \p_\xi \hf )\|_{L^2} \|v \M f\|_{L^2} \\\no
  \lm & \|f\|_{L^1} \|\M f\|^2_{L^2_1}.
\end{align}

Then it follows
\begin{align}\label{I-21-M}
  |I_{21}| \le |I_{211}|+|I_{212}| \lm \|f\|_{L^1} \|\M f\|^2_{L^2_1}.
\end{align}

Concerning the term $I_{22}$, we deduce that
\begin{align}\label{I-22-M}
  |I_{22}| \lm & \iint b(\cos \theta) \sin^2 \frac{\theta}{2} \left|\hf(\xi^-)\right|
                   \left| \M(\xi^+) \Big(\p_\xi\hf\Big)(\xi^+) \right|
                   \left|\p_\xi \Big(\M(\xi) \hf(\xi) \Big)\right|\ d\xi d\sigma \\\no
  \lm & \|f\|_{L^1}\ \|\M f\|^2_{L^2_1}.
\end{align}

As for $I_{23}$, we use the Taylor expansion up to order 2 to get,
\begin{align}\label{Taylor2-1-M}
  \left(\p_\xi \M\right)(\xi) - \left(\p_\xi \M\right)(\xi^+)
= (\xi-\xi^+) \cdot \left(\p^2_{\xi \xi} \M\right)(\xi^+)
 + \int_0^1 (1-\tau) (\xi-\xi^+) \otimes (\xi-\xi^+) : \left(\p^3_{\xi \xi \xi} \M\right)(\xi_\tau) d\tau
\end{align}
with $\tau \in [0,1]$ and $\xi_\tau = (1-\tau) \xi^+ + \tau \xi$. Then we can rewrite $I_{23}$ as $I_{23}=I_{231}+I_{232}$, correspondingly.

A similar process as above ensures that,
\begin{align}
  |I_{231}| \lm & \iint b(\cos\theta) \sin^2 \frac{\theta}{2}
                 \left|\hf(\xi^-)\right| |\xi|\la\xi^+\ra^{-2} \left|\M(\xi^+) \hf(\xi^+)\right|
                 \left|\p_\xi \Big(\M(\xi) \hf(\xi) \Big)\right|\ d\xi d\sigma \\\no
  \lm & \|\hf\|_{L^\infty}\ \|\M\hf\|_{L^2}\ \|v \M f\|_{L^2_1} \\\no
  \lm & \|f\|_{L^1}\ \|\M f\|^2_{L^2_1},
\end{align}
and
\begin{align}
  |I_{232}| \lm & \iiint_0^1 b(\cos\theta) \sin^2 \frac{\theta}{2}
                 \left|\hf(\xi^-)\right| |\xi|^2\la\xi_\tau\ra^{-3} \left|\M(\xi^+) \hf(\xi^+)\right|
                 \left|\p_\xi \Big(\M(\xi) \hf(\xi) \Big)\right|\ d\xi d\sigma \\\no
  \lm & \|\hf\|_{L^\infty}\ \|\M\hf\|_{L^2}\ \|v \M f\|_{L^2_1} \\\no
  \lm & \|f\|_{L^1}\ \|\M f\|^2_{L^2_1}.
\end{align}
The two estimates imply that,
\begin{align}\label{I-23-M}
  |I_{23}| \le |I_{231}|+|I_{232}| \lm \|f\|_{L^1}\ \|\M f\|^2_{L^2_1}.
\end{align}

Thus, from (\ref{I-21-M}), (\ref{I-22-M}), and (\ref{I-23-M}), we have,
\begin{align}
  |I_2| \lm \|f\|_{L^1}\ \|\M f\|^2_{L^2_1}.
\end{align}

Combining with (\ref{I-1-M}), this completes the proof of Proposition \ref{Commutator-1-M}.
\end{proof}

\begin{prop}\label{Commutator-2-M}
  Suppose that $0<s<1$. For a suitable function $f$, we have,
  \begin{align}
    | \la v \otimes v\M Q(f,f) - Q(f,\ v \otimes v\M f),\ v \otimes v\M f \ra |
    \lm \Big(\|f\|_{L^1} + \|f\|_{L^1_1} + \|f\|_{L^1_2}\Big) \|\M f\|^2_{L^2_2}.
  \end{align}
\end{prop}

\begin{proof}
  The proof is similar as that of Proposition \ref{Commutator-2}, if we replace Lemma \ref{estimates for derivations} by Lemma \ref{estimates for derivations-M}, and substitute (\ref{Gd inequ}) for (\ref{Md inequ}). Firstly we write that,
\begin{align}
  & \left \la v\otimes v \M Q(f,f)-Q(f,v\otimes v \M f),\ v\otimes v \M f \right \ra \\
= & C \iint b(\cos\theta) \left\{ \p^2_{\xi\xi} \left( \M(\xi) \hf(\xi^-) \hf(\xi^+) \right)
                                - \hf(\xi^-) \left(\p^2_{\xi\xi}(\M\hf)\right)(\xi^+)
                          \right\}\
      \overline{\left(\p^2_{\xi\xi}(\M\hf)\right)(\xi)}\ d\xi d\sigma \nonumber \\
=& C \iint b(\cos\theta) \left\{\left(\p^2_{\xi\xi} \hf\right)(\xi^-)\Big(\frac{\p\xi^-}{\p\xi}\Big)^2
                          + \left(\p_\xi \hf\right)(\xi^-) \left(\frac{\p^2\xi^-}{\p\xi \p\xi}\right)
                         \right\} \M(\xi)\hf(\xi^+)\
      \overline{\left(\p^2_{\xi\xi}(\M\hf)\right)(\xi)}\ d\xi d\sigma \nonumber \\
& +2C \iint b(\cos\theta) \left(\p_\xi \hf\right)(\xi^-)\ \frac{\p\xi^-}{\p\xi}\
            \p_\xi \left(\M(\xi) \hf(\xi^+)\right)
      \overline{\left(\p^2_{\xi\xi}(\M\hf)\right)(\xi)}\ d\xi d\sigma \nonumber \\
& +C \iint b(\cos\theta) \hf(\xi^-)
        \left\{\p^2_{\xi\xi} \left(\M(\xi)\hf(\xi^+) \right)-\left(\p^2_{\xi\xi}(\M\hf)\right)(\xi^+)
        \right\}\
      \overline{\left(\p^2_{\xi\xi}(\M\hf)\right)(\xi)}\ d\xi d\sigma \nonumber \\
\triangleq & II_1 + II_2 + II_3. \nonumber
\end{align}

Noticing the fact
\begin{align}
  \left| \frac{\p^2 \xi^-}{\p\xi \p\xi} \right| =0,
\end{align}
we have
\begin{align}\label{II-1-M}
  |II_1| \lm & \iint b(\cos\theta) \sin^4\frac{\theta}{2}
                 \left|\left(\p^2_{\xi\xi}\hf\right)(\xi^-)\right|\
                 \left|\M(\xi^+)\hf(\xi^+)\right|\
                 \left|\left(\p^2_{\xi\xi}(\M\hf)\right)(\xi)\right|\ d\xi d\sigma \\\no
  \lm & \|\p^2_{\xi\xi}\hf\|_{L^\infty} \int b(\cos\theta) \sin^4\frac{\theta}{2} d\sigma\
        \|\M\hf\|_{L^2} \|\p^2_{\xi\xi}(\M\hf)\|_{L^2} \\\no
  \lm & \|\p^2_{\xi\xi}\hf\|_{L^1}\ \|\M f\|^2_{L^2_2}.
\end{align}

Above, we have used the following fact, in the last inequality,
\begin{align}\label{M-D2-f}
  \|\M (\p^2_{\xi\xi}\hf)\|_{L^2}
  \le & \|\p^2_{\xi\xi}(\M\hf)\|_{L^2} + \|(\p^2_{\xi\xi}\M)\hf\|_{L^2}
     + 2\|(\p_\xi\M) (\p_\xi\hf)\|_{L^2} \\\no
  \lm & \|\p^2_{\xi\xi}(\M\hf)\|_{L^2} + \|\la\xi\ra^{-2}\M\hf\|_{L^2}
      + \|\la\xi\ra^{-1}\M (\p_\xi\hf)\|_{L^2} \\\no
  \lm & \|\p^2_{\xi\xi}(\M\hf)\|_{L^2} + \|\M\hf\|_{L^2} + \|\M (\p_\xi\hf)\|_{L^2} \\\no
  \lm & \|\la v \ra^2 (\M f)\|_{L^2}.
\end{align}

Considering the term $II_2$, we reformulate it as
\begin{align}
  II_2 = & C \iint b(\cos\theta)\ \frac{\p\xi^-}{\p\xi}\left(\p_\xi \hf\right)(\xi^-)
              \Big(\p_\xi\M\Big)(\xi)\hf(\xi^+)\
              \overline{\left(\p^2_{\xi\xi}(\M\hf)\right)(\xi)}\ d\xi d\sigma \\\no
  &+C \iint b(\cos\theta)\ \frac{\p\xi^-}{\p\xi}\left(\p_\xi \hf\right)(\xi^-)
              \M(\xi)\Big(\p_\xi\hf\Big)(\xi^+)\frac{\p\xi^+}{\p\xi}\
              \overline{\left(\p^2_{\xi\xi}(\M\hf)\right)(\xi)}\ d\xi d\sigma \\\no
  \triangleq & II_{21} + II_{22}.
\end{align}

Then we can deduce that
\begin{align}
  |II_{21}| \lm& \iint b(\cos\theta)\sin^2\th2\ \left|\left(\p_\xi \hf\right)(\xi^-)\right|\
                  \la\xi\ra^{-1} \left|\M(\xi^+)\hf(\xi^+)\right|
                  \left|\left(\p^2_{\xi\xi}(\M\hf)\right)(\xi)\right|\ d\xi d\sigma  \\\no
  \lm& \|\p_\xi \hf\|_{L^\infty}\ \|\M\hf\|_{L^2} \|\p^2_{\xi\xi}(\M\hf)\|_{L^2}  \\\no
  \lm& \|f\|_{L^1_1}\ \|\M f\|^2_{L^2_2},
\end{align}
and
\begin{align}
  |II_{22}| \lm& \iint b(\cos\theta)\sin^2\th2 \cos^2\th2\
                  \left|\left(\p_\xi \hf\right)(\xi^-)\right|\
                  \left|\Big(\M (\p_\xi\hf)\Big)(\xi^+)\right|\
                  \left|\left(\p^2_{\xi\xi}(\M\hf)\right)(\xi)\right|\ d\xi d\sigma  \\\no
  \lm& \|\p_\xi \hf\|_{L^\infty}\ \|\M(\p_\xi\hf)\|_{L^2} \|\p^2_{\xi\xi}(\M\hf)\|_{L^2}  \\\no
  \lm& \|f\|_{L^1_1}\ \|\M f\|^2_{L^2_2}.
\end{align}

From the above two estimates it follows
\begin{align}\label{II-2-M}
  |II_2| \le  |II_{21}| + |II_{22}| \lm \|f\|_{L^1_1}\ \|\M f\|^2_{L^2_2}.
\end{align}

We then turn to the estimate for $II_3$, we write
\begin{align}
  II_3 = & C \iint b(\cos\theta) \hf(\xi^-)
       \left\{\p^2_{\xi\xi}\Big(\M(\xi)\hf(\xi^+)\Big) - \Big(\p^2_{\xi\xi}(\M\hf)\Big)(\xi^+)\right\}
       \overline{\left(\p^2_{\xi\xi}(\M\hf)\right)(\xi)}\ d\xi d\sigma \nonumber \\
  = & C \iint b(\cos\theta) \hf(\xi^-) \Big\{(\p^2_{\xi\xi}\M)(\xi)-(\p^2_{\xi\xi}\M)(\xi^+)\Big\}\
              \hf(\xi^+)\ \overline{\left(\p^2_{\xi\xi}(\M\hf)\right)(\xi)}\ d\xi d\sigma \nonumber \\
  &+C \iint b(\cos\theta) \hf(\xi^-)\ \M(\xi) \Big(\p_\xi \hf\Big)(\xi^+)
            \left(\frac{\p^2\xi^+}{\p\xi \p\xi}\right)\
            \overline{\left(\p^2_{\xi\xi}(\M\hf)\right)(\xi)}\ d\xi d\sigma \nonumber \\
  &+C \iint b(\cos\theta) \hf(\xi^-)
            \left\{(\p_\xi\M)(\xi) \left(\frac{\p\xi^+}{\p\xi}\right) - (\p_\xi\M)(\xi^+)\right\}
            \Big(\p_\xi \hf\Big)(\xi^+)\
            \overline{\left(\p^2_{\xi\xi}(\M\hf)\right)(\xi)}\ d\xi d\sigma \nonumber \\
  &+C \iint b(\cos\theta) \hf(\xi^-)
            \left\{\M(\xi) \Big(\frac{\p\xi^+}{\p\xi}\Big)^2 - \M(\xi^+)\right\}
            \left(\p^2_{\xi\xi}\hf\right)(\xi^+)\
            \overline{\left(\p^2_{\xi\xi}(\M\hf)\right)(\xi)}\ d\xi d\sigma \nonumber \\
\triangleq & II_{31} + \Psi + II_{32} + II_{33}. \nonumber
\end{align}

The process of dealing with these above terms is similar as that of Proposition \ref{Commutator-2}, and much simpler. Thus we omit it and give the following estimate,
\begin{align}\label{II-3-M}
  |II_3| \le  |II_{31}| + |\Psi| + |II_{32}| + |II_{33}| \lm \|f\|_{L^1}\ \|\M f\|^2_{L^2_2}.
\end{align}

Combining the estimates (\ref{II-1-M}), (\ref{II-2-M}), and (\ref{II-3-M}) yields the desired result.
\end{proof}

\subsection{Justification for Sobolev regularizing effect}

Recalling the following upper bound for the collision operator (compare \cite{Alex-review}):
\begin{align}
  \|Q(g,f)\|_{H^m_l(\mathbb{R}^3)}
 \lm \|g\|_{L^1_{l^+ +2s}(\mathbb{R}^3)}\ \|f\|_{H^{m+2s}_{(l+2s)^+}(\mathbb{R}^3)},
\end{align}
with $m=-4,\ l=2$ and $0<s<1$, we get
\begin{align}
  \|Q(f,f)\|_{H^{-4}_2(\mathbb{R}^3)}
 \lm \|f\|_{L^1_{2+2s}(\mathbb{R}^3)}\ \|f\|_{H^{-4+2s}_{2+2s}(\mathbb{R}^3)}
 \lm \|f\|_{L^1_{2+2s}(\mathbb{R}^3)}\ \|f\|_{L^1_{2+2s}(\mathbb{R}^3)}.
\end{align}
Let $f \in L^\infty((0,\ T_0);\ L^1_{2+2s} \cap LlogL(\mathbb{R}^3))$ be a weak solution of the Cauchy problem (\ref{BE}), then we take
\begin{align}
  f_1(t,\ \cdot)=\Big(\M \la v\ra^4 \M f \Big)(t,\ \cdot)
  \in L^\infty ([0,T_0];\ H^5_{2+2s}(\mathbb{R}^3))
\end{align}
as the test function. And moreover, a similar argument as that of \cite{Mori-Ukai-Xu-Yang} enables us to assume $f_1 \in C^1 ([0,T_0];\ H^5_{2+2s}(\mathbb{R}^3))$.

Then we obtain the weak formulation:
\begin{align}
  \Big\la \p_t f(t,\ \cdot),\ f_1(t,\ \cdot) \Big\ra
  = \Big\la Q(f,f),\ f_1 \Big\ra.
\end{align}

We compute
\begin{align}
  L.H.S.=& \frac{1}{2}\frac{d}{dt}\|\M f\|^2_{L^2_2} - \la(\p_t\M)f,\ \M f\ra
        - 2\la v(\p_t\M)f,\ v\M f\ra - \la v^2(\p_t\M)f,\ v^2\M f\ra, \\
  R.H.S.=& \la \M Q(f,f),\ (1+2|v|^2+|v|^4) \M f\ra \\
        =& \la Q(f,\M f),\ \M f\ra + 2\la Q(f,v\M f),\ v\M f\ra + \la Q(f,v\otimes v\M f),\ v\otimes v\M f\ra \nonumber \\
         & + \la \M Q(f,f)-Q(f,\M f),\ \M f\ra + 2\la v\M Q(f,f)-Q(f,v\M f),\ v\M f\ra \nonumber \\
         & + \la v\otimes v\M Q(f,f)-Q(f,v\otimes v\M f),\ v\otimes v\M f\ra, \nonumber
\end{align}
then we get the reformulation:
\begin{align}
  &\frac{1}{2}\frac{d}{dt}\|\M f\|^2_{L^2_2} - \la Q(f,\M f),\ \M f\ra - 2\la Q(f,v\M f),\ v\M f\ra
    - \la Q(f,v\otimes v\M f),\ v\otimes v\M f\ra \\
 =& \la(\p_t\M)f,\ \M f\ra + 2\la v(\p_t\M)f,\ v\M f\ra + \la v^2(\p_t\M)f,\ v^2\M f\ra \nonumber \\
  & + \la \M Q(f,f)-Q(f,\M f),\ \M f\ra + 2\la v\M Q(f,f)-Q(f,v\M f),\ v\M f\ra \nonumber \\
  & + \la v\otimes v\M Q(f,f)-Q(f,v\otimes v\M f),\ v\otimes v\M f\ra. \nonumber
\end{align}

In the next, we need to handle with the three terms on the right-hand side, as follows:
\begin{lemm}\label{D_t of M}
  For the terms involving the derivative of the mollifier with respect to time, we have
  \begin{align}
   &|\la(\p_t\M)f,\ \M f\ra| \le \ep \|\M f\|^2_{H^s} + C_\ep \|\M f\|^2_{L^2}, \\
   &|\la v(\p_t\M)f,\ v\M f\ra| \le \ep \|\M f\|^2_{H^s_1} + C_\ep \|\M f\|^2_{L^2_1}, \\
   &|\la v^2(\p_t\M)f,\ v^2\M f\ra| \le \ep \|\M f\|^2_{H^s_2} + C_\ep \|\M f\|^2_{L^2_2}.
  \end{align}
\end{lemm}

\begin{proof}
  By virtue of Lemma \ref{estimates for derivations-M}, we have
  \begin{align}
    \la(\p_t\M)f,\ \M f\ra =\Big\la \frac{N}{2}\log(1+|\xi|^2)\M \hf,\ \M \hf\Big\ra.
  \end{align}

Noticing the fact, with an $\ep>0$,
\begin{align}\label{use3}
  \frac{N}{2}\log(1+|\xi|^2) \le \ep (1+|\xi|^2)^s + C_\ep,
\end{align}
thus, we get
\begin{align}
  |\la(\p_t\M)f,\ \M f\ra|
 \le \ep \Big\la (1+|\xi|^2)^s \M \hf,\ \M \hf\Big\ra + C_\ep \Big\la\M \hf,\ \M \hf\Big\ra
 \le \ep \|\M f\|^2_{H^s} + C_\ep \|\M f\|^2_{L^2}.
\end{align}

After a few calculations, we have,
\begin{align}
  &\p_\xi(\p_t\M\hf) \le \p_\xi\Big(\frac{N}{2}\log(1+|\xi|^2)\Big) \M\hf
                          + \frac{N}{2}\log(1+|\xi|^2) \p_\xi(\M\hf), \\
  &\p^2_{\xi\xi}(\p_t\M\hf) \le \p^2_{\xi\xi}\Big(\frac{N}{2}\log(1+|\xi|^2)\Big) \M\hf
                              + 2\p_\xi\Big(\frac{N}{2}\log(1+|\xi|^2)\Big) \p_\xi(\M\hf)
                              + \frac{N}{2}\log(1+|\xi|^2) \p^2_{\xi\xi}(\M\hf).
\end{align}

Observing that,
\begin{align}
  \left|\p_\xi\Big(\frac{N}{2}\log(1+|\xi|^2)\Big)\right| \le C,\ \
  \left|\p^2_{\xi\xi}\Big(\frac{N}{2}\log(1+|\xi|^2)\Big)\right| \le C,
\end{align}
we can derive the latter two results by using of (\ref{use3}).
\end{proof}

Now we resume to the proof of Theorem \ref{Thm Sobolev}. From Lemma \ref{Coercivity lemma}, it's easy to check that
\begin{align}
  &- \la Q(f,\M f),\ \M f\ra \ge c_f\|\M f\|^2_{H^s} - C\|f\|_{L^1}\ \|\M f\|^2_{L^2},
  \label{coercivity-0-M}  \\
  &- 2\la Q(f,v\M f),\ v\M f\ra \ge c_f\|\M f\|^2_{H^s_1}-C\|f\|_{L^1}\ \|\M f\|^2_{L^2_1},
  \label{coercivity-1-M}\\
  &- \la Q(f,v\otimes v\M f),\ v\otimes v\M f\ra
  \ge c_f\|\M f\|^2_{H^s_2} - C\|f\|_{L^1}\ \|\M f\|^2_{L^2_2}. \label{coercivity-2-M}
\end{align}

Combining these estimates, Lemma \ref{D_t of M} and Propositions \ref{Commutator-0-M}--\ref{Commutator-2-M}, we obtain that
\begin{align}\label{use2}
   \frac{d}{dt}\|\M f\|^2_{L^2_2} + c_f \|\M f\|^2_{H^s_2}
\le \ep \|\M f\|^2_{H^s_2} + C_\ep \|\M f\|^2_{L^2_2}
     +C \Big(\|f\|_{L^1} + \|f\|_{L^1_1} + \|\M f\|_{L^2_2}\Big)\ \|\M f\|^2_{L^2_2}.
\end{align}

Recalling the conservational properties for the Boltzmann equation implies that
\begin{align}\label{conservation}
  \|f\|_{L^1} + \|f\|_{L^1_1} + \|\M f\|_{L^2_2} \lm \|f_0\|_{L^1_2},\quad c_f \ge c_{f_0}>0,
\end{align}
and by choosing $\ep<c_f$, we get
\begin{align}
\frac{d}{dt}\|\M f\|^2_{L^2_2} \le C\|\M f\|^2_{L^2_2},
\end{align}
from which it follows
\begin{align}
  \|(\M f)(t)\|^2_{L^2_2} \le  e^{Ct}\ \|\M(0) f_0\|^2_{L^2_2}.
\end{align}

Since
\begin{align}
  &\|(\M f)(t)\|^2_{L^2_2} = \|(1-\delta \Delta)^{-N_0} f(t)\|^2_{H^{Nt-2}_2},\\
  &\|\M(0) f_0\|^2_{L^2_2} = \|(1-\delta \Delta)^{-N_0} f(0)\|^2_{H^{-2}_2} \le C\|f_0\|^2_{L^1_2},
\end{align}
due to the embedding $L^1_2(\mathbb{R}^3) \subset H^{-2}_2(\mathbb{R}^3)$, then we obtain
\begin{align}
  \|(1-\delta \Delta)^{-N_0} f(t)\|^2_{H^{Nt-2}_2} \le C e^{Ct}\ \|f_0\|^2_{L^1_2},
\end{align}
where $C>0$ is independent of $\delta$. Taking limit $\delta \to 0$, we get, finally, for $t\in [0,T_0]$,
\begin{align}
  \|f(t)\|^2_{H^{Nt-2}_2} \le C e^{Ct}\ \|f_0\|^2_{L^1_2}.
\end{align}

As $N$ can be chosen arbitrarily large for any given $t>0$, we conclude that
\begin{align}
  f(t) \in H^{+\infty}_2(\mathbb{R}^3).
\end{align}

This completes the whole proof of Theorem \ref{Thm Sobolev}.

\section{Completion of the proof of Gevrey regularity}

Since for any $t_0>0$, the weak solution satisfies $f \in L^\infty([t_0,\ T_0];\ H^2_2(\mathbb{R}^3))$, then $f$ solves the following Cauchy problem:
\begin{align}\label{BE-var}
  \left\{
    \begin{array}{l}
      f_t(t,v)=Q(f,f)(v),\ t\in (t_0,T], ~  v \in \mathbb{R}^3,\\
      f|_{t=t_0}=f(t_0,\cdot) \in H^2_2(\mathbb{R}^3).
    \end{array}
  \right.
\end{align}

When considering the Gevrey regularizing effect, we may assume $t_0=0$ by translation. Thus we state the result:
\begin{theo}\label{Thm Gevrey}
  Suppose the initial datum $f_0 \in L^1_{2+2s} \cap H^2_2(\mathbb{R}^3)$. Let $f \in L^\infty([0,\ T_0];\ L^1_2 \cap H^2_2(\mathbb{R}^3))$ be a non-negative weak solution of the Cauchy problem of the Boltzmann equation (\ref{BE}) for some $T_0>0$, then
\begin{itemize}
      \item[i)] for the mild singularity case $0< s < \frac{1}{2}$, there exists $0<T_*  \le T_0$ such that $f(t,\cdot) \in G^\frac{1}{2\alpha}(\mathbb{R}^3)$ for any $0<\alpha<s$ and $0<t \le T_*$, more precisely, there exists $c_0>0$ such that
          \begin{align}
            e^{c_0 t \la D_v \ra^{2\a}} f \in L^\infty([0,T_*]);\ L^2_2(\mathbb{R}^3));
          \end{align}

      \item[ii)] in the critical case of $s = \frac{1}{2}$, there exists $0<T_* \le T_0$ such that $f(t,\cdot) \in G^{\frac{3}{2}+\ep}(\mathbb{R}^3)$ for any $\ep>0$ and $0<t \le T_*$, moreover, there exists $c_0>0,\ \ep'>0$ such that
          \begin{align}
            e^{c_0 t \la D_v \ra^{\frac{2(1-\ep')}{3}}} f
            \in L^\infty([0,T_*]);\ L^2_2(\mathbb{R}^3));
          \end{align}

      \item[iii)] for the strictly strong singularity case $s \ge \frac{1}{2}$, there exists $0<T_* \le T_0$ such that $f(t,\cdot) \in G^\frac{3}{2}(\mathbb{R}^3)$ for any $0<t \le T_*$, in precise, there exists $c_0>0$ such that
          \begin{align}
            e^{c_0 t \la D_v \ra^{\frac{2}{3}}} f \in L^\infty([0,T_*]);\ L^2_2(\mathbb{R}^3)).
          \end{align}
    \end{itemize}

\end{theo}

\begin{rema}
  By virtue of Theorem 1.2 of \cite{Desvillettes} describing the propagation of Gevrey regularity, the above theorem can lead to the Gevrey smoothing effect in global time showed in Theorem \ref{main result}.
\end{rema}

Let $f \in L^\infty([0,\ T_0];\ L^1_2 \cap H^2_2(\mathbb{R}^3))$ be a weak solution of the Cauchy problem (\ref{BE}), then recall the following upper bound for the collision operator (compare \cite{Alex-review}):
\begin{align}
  \|Q(g,f)\|_{H^m_l(\mathbb{R}^3)}
 \lm \|g\|_{L^1_{l^+ +2s}(\mathbb{R}^3)}\ \|f\|_{H^{m+2s}_{(l+2s)^+}(\mathbb{R}^3)},
\end{align}
with $m=l=0$ and $0<s<1$, hence we have
\begin{align}
  \|Q(f,f)\|_{L^2(\mathbb{R}^3)}
 \lm \|f\|_{L^1_{2s}(\mathbb{R}^3)}\ \|f\|_{H^{2s}_{2s}(\mathbb{R}^3)}
 \lm \|f\|_{L^1_2(\mathbb{R}^3)}\ \|f\|_{H^2_2(\mathbb{R}^3)},
\end{align}
which implies that $Q(f,f) \in L^\infty ([0,T_0];\ L^2(\mathbb{R}^3))$. Therefore we need to choose a test function $\phi \in C^1([0,T_0];\ L^2(\mathbb{R}^3))$ to make sense $\la Q(f,f),\ \phi \ra$.

We choose the mollified weak solution
\begin{align}
  \widetilde{f}(t,\ \cdot)=\Big(\G \la v \ra ^4 \G f \Big)(t,\ \cdot) \in L^\infty ([0,T_0];\ H^2(\mathbb{R}^3)).
\end{align}
Furthermore, we suppose that $\widetilde{f}(t,\ \cdot)\in C^1 ([0,T_0];\ H^2(\mathbb{R}^3))$. Then we get
\begin{align}
  \Big\la \p_t f(t,\ \cdot),\ \widetilde{f}(t,\ \cdot) \Big\ra
  = \Big\la Q(f,f),\ \widetilde{f} \Big\ra
  = \Big\la \G Q(f,f),\ (1+2|v|^2+|v|^4) \G f \Big\ra,
\end{align}
which yields the reformulation:
\begin{align}
  &\frac{1}{2}\frac{d}{dt}\|\G f\|^2_{L^2_2} - \la Q(f,\G f),\ \G f\ra - 2\la Q(f,v\G f),\ v\G f\ra
    - \la Q(f,v\otimes v\G f),\ v\otimes v\G f\ra \\
 =& \la(\p_t\G)f,\ \G f\ra + 2\la v(\p_t\G)f,\ v\G f\ra + \la v^2(\p_t\G)f,\ v^2\G f\ra \nonumber \\
  & + \la \G Q(f,f)-Q(f,\G f),\ \G f\ra + 2\la v\G Q(f,f)-Q(f,v\G f),\ v\G f\ra \nonumber \\
  & + \la v\otimes v\G Q(f,f)-Q(f,v\otimes v\G f),\ v\otimes v\G f\ra, \nonumber
\end{align}

Now it remains to estimate the three terms on the right-hand side, as follows:
\begin{lemm}\label{D_t of G}
  For the terms involving the derivative of the mollifier with respect to time, we have
  \begin{align}
   &|\la(\p_t\G)f,\ \G f\ra| \lm \|\G f\|^2_{H^\a}, \\
   &|\la v(\p_t\G)f,\ v\G f\ra| \lm \|\G f\|^2_{H^\a_1}, \\
   &|\la v^2(\p_t\G)f,\ v^2\G f\ra| \lm \|\G f\|^2_{H^\a_2}.
  \end{align}
\end{lemm}

\begin{proof}
  Due to the Plancherel formula, we can deduce directly the first result by using Lemma \ref{estimates for derivations}.

  As for the second result, we can write that
  \begin{align}
    &|\la v(\p_t\G)f,\ v\G f\ra \\
   =&\left| \int \left\{\p_\xi
                  \left(c_0 \la\xi\ra^{2\a}\G(\xi)\frac{1}{1+\delta e^{c_0t\la\xi\ra^{2\a}}} \hf(\xi)
                  \right)
                 \right\} \overline{(v\G f)^{\wedge}(\xi)} d\xi
     \right| \nonumber \\
  \lm & \left| \int \la\xi\ra^{2\a} \p_\xi \left(\G(\xi) \hf(\xi)\right)
               \overline{(v\G f)^{\wedge}(\xi)} d\xi \right| \nonumber \\
     &+ \int \left| \p_\xi \left(\la\xi\ra^{2\a}\frac{1}{1+\delta e^{c_0t\la\xi\ra^{2\a}}}\right)
             \right|\ \left|\G(\xi) \hf(\xi)\right|\
             \left|(v\G f)^{\wedge}(\xi)\right|\ d\xi  \nonumber \\
  \lm &\|\G f\|^2_{H^\a_1}, \nonumber
  \end{align}
  in view of the estimate
  \begin{align}
    \left| \p_\xi \left(\la\xi\ra^{2\a}\frac{1}{1+\delta e^{c_0t\la\xi\ra^{2\a}}}\right) \right|
   \lm \la\xi\ra^{2\a},\ for\ \a<\frac{1}{2}.
  \end{align}

  By a similar but more slightly complicate scheme, the fact
  \begin{align}
    \left| \p^2_{\xi\xi} \left(\la\xi\ra^{2\a}\frac{1}{1+\delta e^{c_0t\la\xi\ra^{2\a}}}\right) \right|
   \lm \la\xi\ra^{2\a}
  \end{align}
  leads to the last result.
\end{proof}

Now we resume to the proof of Theorem \ref{Thm Gevrey}. From Lemma \ref{Coercivity lemma}, it's easy to check that
\begin{align}
  &- \la Q(f,\G f),\ \G f\ra \ge c_f\|\G f\|^2_{H^s} - C\|f\|_{L^1}\ \|\G f\|^2_{L^2},
  \label{coercivity-0}  \\
  &- 2\la Q(f,v\G f),\ v\G f\ra \ge c_f\|\G f\|^2_{H^s_1}-C\|f\|_{L^1}\ \|\G f\|^2_{L^2_1},
  \label{coercivity-1}\\
  &- \la Q(f,v\otimes v\G f),\ v\otimes v\G f\ra
  \ge c_f\|\G f\|^2_{H^s_2} - C\|f\|_{L^1}\ \|\G f\|^2_{L^2_2}. \label{coercivity-2}
\end{align}

Combining these estimates, Lemma \ref{D_t of G} and Propositions \ref{Commutator-0}, \ref{Commutator-1}, \ref{Commutator-2}, we can infer that
\begin{align}\label{use2}
   \frac{d}{dt}\|\G f\|^2_{L^2_2} + c_f \|\G f\|^2_{H^s_2}
\le &C\|\G f\|^2_{H^\a_2} + C \|\G f\|^3_{L^2_2}
     +C \Big(\|f\|_{L^1} + \|f\|_{L^1_1} + \|\G f\|_{L^2_2}\Big)\ \|\G f\|^2_{H^\a_2} \\
   & +C \Big(\|f\|_{L^1} + \|\G f\|_{L^2_2}\Big) \|\G f\|^2_{H^{(3\a-\frac{1}{2})^+}_2}. \nonumber
\end{align}

Recalling that the conservational properties for the Boltzmann equation implies,
\begin{align}\label{conservation}
  \|f\|_{L^1} + \|f\|_{L^1_1} + \|\G f\|_{L^2_2} \lm \|f_0\|_{L^1_2},\quad c_f \ge c_{f_0}>0,
\end{align}
we then have
\begin{align}
    \frac{d}{dt}\|\G f\|^2_{L^2_2} + c_{f_0} \|\G f\|^2_{H^s_2}
\le & C_{f_0}\|\G f\|^2_{H^\a_2} + C \|\G f\|^3_{L^2_2}
     + C_{f_0} \|\G f\|^2_{H^{(3\a-\frac{1}{2})^+}_2} \\
    &+ C\|\G f\|_{L^2_2}\ \|\G f\|^2_{H^\a_2}
     + C\|\G f\|_{L^2_2}\ \|\G f\|^2_{H^{(3\a-\frac{1}{2})^+}_2}. \nonumber
\end{align}

\vspace*{1.5em}
{\bf -Case $\frac{1}{2}<s<1$:}

Take $\a=\frac{1}{3}$, then $(3\a-\frac{1}{2})^+=\frac{1}{2}<s$, and we have
\begin{align}
    \frac{d}{dt}\|\G f\|^2_{L^2_2} + c_{f_0} \|\G f\|^2_{H^s_2}
\le C_{f_0}\|\G f\|^2_{H^\frac{1}{2}_2} + C \|\G f\|^3_{L^2_2}
    + C\|\G f\|_{L^2_2}\ \|\G f\|^2_{H^\frac{1}{2}_2}.
\end{align}

Thanks to the following interpolation inequalities:
\begin{align}
  &\|\G f\|^2_{H^\frac{1}{2}_2}
  \le \rho \|\G f\|^2_{H^s_2} + \rho^{-\frac{1/3}{s-1/3}}\|\G f\|^2_{L^2_2}, \\
  &\|\G f\|_{L^2_2}\ \|\G f\|^2_{H^\frac{1}{2}_2}
  \le \rho \|\G f\|^2_{H^s_2} + C_\rho \|\G f\|^{2+\frac{2s}{2s-1}}_{L^2_2},
\end{align}
and noticing the simple fact $2+\frac{2s}{2s-1}>3$, it follows that, by choosing $2\rho=c_{f_0}/2$,
\begin{align}
  \frac{d}{dt}\|\G f\|^2_{L^2_2} + \frac{c_{f_0}}{2} \|\G f\|^2_{H^s_2} \le C\|\G f\|^2_{L^2_2} + C \|\G f\|^{2+\frac{2s}{2s-1}}_{L^2_2}.
\end{align}

This is an ordinary differential equation of Bernoulli type including an extra term on the left-hand. Putting $g=e^{-Ct}\|\G f\|^2_{L^2_2}$, we get
\begin{align}
  \frac{d}{dt}g \le Ce^{\widetilde{C}t} g^{1+\frac{s}{2s-1}},
\end{align}
with $\widetilde{C}=\frac{s}{2s-1}C$. Therefore, we can deduce that,
\begin{align}
  g(t) \le
 \frac{g(0)}{\left\{1+C\Big(1-e^{\widetilde{C}t}\Big)g(0)^\frac{s}{2s-1}\right\}^\frac{2s-1}{s}},
\end{align}
which yields, for $0<\delta<1$,
\begin{align}
  \|\G f\|^2_{L^2_2} \le
 \frac{e^{Ct} \|f_0\|^2_{L^2_2}}
 {\left\{1+C\Big(1-e^{\widetilde{C}t}\Big)\|f_0\|^\frac{2s}{2s-1}_{L^2_2}\right\}^\frac{2s-1}{s}}.
\end{align}

We choose $T_*\in (0,T_0]$ sufficiently small such that
\begin{align}
  \left\{1+C\Big(1-e^{\widetilde{C}t}\Big)\|f_0\|^\frac{2s}{2s-1}_{L^2_2}\right\}^\frac{2s-1}{s}
  \ge C_0 >0, \quad t\in [0,\ T_*].
\end{align}

Taking limit $\delta \rightarrow 0$, we obtain for $t\in [0,\ T_*]$:
\begin{align}
  \|e^{c_0 t \la D_v\ra^{2/3}} f\|^2_{L^\infty([0,T_*];\ L^2_2(\mathbb{R}^3))}
 \le C^{-1}_0 e^{CT_*}\|f_0\|^2_{L^2_2(\mathbb{R}^3)}.
\end{align}

This completes the justification for the case $s\in (1/2,\ 1)$.

\vspace*{1.5em}
{\bf -Case $s=1/2$:}

Given $\eta \in (0,\frac{1}{2})$, and by taking $\a=\frac{1-\eta}{3} \in (\frac{1}{6},\frac{1}{3})$, we have $(3\a-\frac{1}{2})^+=\frac{1}{2}-\eta \in (0,\frac{1}{2})$. Reasoning along exactly the same lines as above, we can state the whole proof in this case $s=1/2$ and obtain the Gevrey smoothing effect in the space $G^{\frac{3}{2(1-\eta)}}$.

\vspace*{2em}
{\bf -Case $0<\a<s<1/2$:}

Combining the coercivity estimates (\ref{coercivity-0})-(\ref{coercivity-2}), Lemma \ref{D_t of G}, Remarks \ref{Remark-order-1} and \ref{Remark-order-2}, we can get, corresponding to (\ref{use2}),
\begin{align}
   \frac{d}{dt}\|\G f\|^2_{L^2_2} + c_f \|\G f\|^2_{H^s_2}
\le C\|\G f\|^2_{H^\a_2}
     + C \Big(\|f\|_{L^1} + \|f\|_{L^1_1} + \|\G f\|_{L^2_2}\Big)\ \|\G f\|^2_{H^\a_2}.
\end{align}

Applying the conservation laws (\ref{conservation}) and interpolation inequality, we can obtain,
\begin{align}
  \frac{d}{dt}\|\G f\|^2_{L^2_2} + \frac{c_{f_0}}{2} \|\G f\|^2_{H^s_2} \le C\|\G f\|^2_{L^2_2} + C \|\G f\|^{2+\frac{s}{s-\a}}_{L^2_2}.
\end{align}

After a few calculations, by choosing $T_*\in(0,T_0]$ small enough and taking limit $\delta \rightarrow 0$, we get finally, for $t\in [0,\ T_*]$:
\begin{align}
  \|e^{c_0 t \la D_v\ra^{2\a}} f\|^2_{L^\infty([0,T_*];\ L^2_2(\mathbb{R}^3))}
 \le C^{-1}_0 e^{C T_*}\|f_0\|^2_{L^2_2(\mathbb{R}^3)},
\end{align}
which leads to the conclusion in the case $0<\a<s<1/2$ and thus, completes the whole proof of Theorem \ref{Thm Gevrey}.

\bigskip
\noindent\textbf{Acknowledgements.}  This work was partially supported by ...

\phantomsection
\addcontentsline{toc}{section}{\refname}


\begin{thebibliography}{99}

\bibitem{Alex-review} R. Alexandre, \textit{A review of Boltzmann equation with singular kernels}, {Kinet. Relat. Mod.}, {\bf 2(4)} (2009), 551-646.

\bibitem{Entropy} R. Alexandre, L. Desvillettes, C. Villani and B. Wennberg, \textit{Entropy dissipation and long-range interactions}, {Arch. Ration. Mech. Anal.} {\bf 152} (2000), 327-355.

\bibitem{FiveGroup-regulariz} R. Alexandre, Y. Morimoto, S. Ukai, C.-J. Xu and T. Yang, \textit{Regularizing effect and local existence for non-cutoff Boltzmann equation}, {Arch. Rational Mech. Anal.}, {\bf 198} (2010), 39-123.

\bibitem{FiveGroup-III} R. Alexandre, Y. Morimoto, S. Ukai, C.-J. Xu and T.Yang, \textit{Boltzmann equation without angular cutoff in the whole space: Qualitative properties of solutions}, {Arch. Rational Mech. Anal.}, {\bf 202(2)} (2011), 599-661.

\bibitem{Chen-He-1} Y. Chen and L. He, \textit{Smoothing estimates for Boltzmann equation with full-range interactions: Spatially homogeneous case}, {Arch. Rational Mech. Anal.}, {\bf 201(2)} (2011), 501-548.

\bibitem{Desvillettes-2001} L. Desvillettes, \textit{Boltzmann's Kernel and the Spatially Homogeneous Boltzmann Equation}, {Riv. di Mat. Parma}, {\bf 6(4)} (2001), 1-22.

\bibitem{Desvillettes} L. Desvillettes, G. Furiolo and E. Terraneo, \textit{Propagation of Gevrey regularity for solutions of the Boltzmann equation for Maxwellian molecules}, {Trans. Amer. Math. Soc.}, {\bf 361} (2009), 1731-1747.

\bibitem{Desvillettes2} L. Desvillettes and B. Wennberg, \textit{Smoothness of the solution of the spatially homogeneous Boltzmann equation without cutoff}, {Comm. Partial Differential Equations}, {\bf 29(1-2)} (2005), 133-155.

\bibitem{Glangetas} L. Glangetas and M. Najeme, \textit{Analytical regularizing effect for the radial homogeneous Boltzmann equation}, {to appear in Kinet. Relat. Mod.} (2012), {Preprint \href{http://arxiv.org/abs/1205.6200v2}{arXiv: 1205.6200v2.pdf}}.

\bibitem{Gressman} P.T. Gressman and R.M. Strain, \textit{Global classical solutions of the Boltzmann equation without angular cut-off}, {J. Amer. Math. Soc.}, {\bf 24(3)}(2011), 771-847.

\bibitem{GuoBoLing} B.L. Guo, \textit{Viscosity elimination method and the viscosity of difference scheme}, {Chinese Sci. Publisher}, (2004).

\bibitem{Huo} Z.H. Huo, Y. Morimoto, S. Ukai and T. Yang, \textit{Regularity of solutions for spatially homogeneous Boltzmann equation without Angular cutoff}, {Kinet. Relat. Mod.}, {\bf 1} (2008), 453-489.

\bibitem{Lekrine-Xu-09} N. Lekrine and C.-J. Xu, \textit{Gevrey regularizing effect of the Cauchy problem for non-cutoff homogeneous Kac's equation}, {Kinet. Relat. Mod.}, {\bf 2(4)} (2009), 647-666.

\bibitem{Morimoto-Starov-GS} N. Lerner, Y. Morimoto, K. Pravda-Starov and C.-J. Xu, \textit{Gelfand-Shilov smoothing properties of the radially symmetric spatially homogeneous Boltzmann equation without angular cutoff}, {Preprint \href{http://arxiv.org/abs/1212.4712}{arXiv: 1212.4712.pdf}}.

\bibitem{Mori-Ukai} Y. Morimoto and S. Ukai, \textit{Gevrey smoothing effect of solutions for spatially homogeneous nonlinear Boltzmann equation without angular cutoff}, {J. Pseudo-Differ. Oper. Appl.}, {\bf 1} (2010), 139-159.

\bibitem{Mori-Ukai-Xu-Yang} Y. Morimoto, S. Ukai, C.-J. Xu and T. Yang, \textit{Regularity of solutions to the spatially homogeneous Boltzmann equation without angular cutoff}, {Discrete Contin. Dyn. Syst.}, {\bf 24} (2009), 187-212.

\bibitem{Mori-Xu-09} Y. Morimoto and C.-J. Xu, \textit{Ultra-analytic effect of Cauchy problem for a class of kinetic equations}, {J. Differential Equations}, {\bf 247} (2009), 596-617.

\bibitem{Nirenberg} L. Nirenberg, \textit{On Elliptic Partial Differential Equations}, {Springer Berlin Heidelberg},  {Springer Berlin Heidelberg}, {C.I.M.E. Summer Schools Series {\bf 17} (2011)}, 1-48.

\bibitem{Ukai-84} S. Ukai, \textit{Local solutions in Gevrey classes to the nonlinear Boltzmann equation without cutoff}, {Japan J. Appl. Math.}, {\bf 1(1)} (1984), 141-156.

\bibitem{Villani} C. Villani, \textit{A review of mathematical topics in collisional kinetic theory}, {In: Handbook of Fluid Mechanics}, (2002).

\bibitem{Zhang-Yin} T.-F. Zhang and Z. Yin, \textit{Gevrey regularity of spatially homogeneous Boltzmann equation without cutoff}, {J. Differential Equations}, {\bf 253(4)} (2012), 1172-1190.

\bibitem{Zhang-Yin-2} T.-F. Zhang and Z. Yin, \textit{Gevrey Regularity for Solutions of the Non-Cutoff Boltzmann Equation: Spatially Inhomogeneous Case}, {Preprint \href{http://arxiv.org/abs/1304.2971}{arXiv: 1304.2971.pdf}}.

\bibitem{Zhang-Yin-3} T.-F. Zhang and Z. Yin, \textit{Gevrey Smoothing Effect of Solutions to Non-Cutoff Boltzmann Equation for Soft Potential with Mild and Critical Singularity}, {Preprint \href{http://arxiv.org/abs/1304.4388}{arXiv: 1304.4388.pdf}}.

\end{thebibliography}

\end{document}